\documentclass{aims}
\usepackage{amsmath,amssymb}
\usepackage{subeqnarray}
\usepackage{mathrsfs}
  \usepackage{paralist}
  \usepackage{graphics} 
  \usepackage{epsfig} 
 \usepackage[colorlinks=true]{hyperref}
\hypersetup{urlcolor=blue, citecolor=red}

\def\currentvolume{X}
 \def\currentissue{X}
  \def\currentyear{200X}
   \def\currentmonth{XX}
    \def\ppages{X--XX}

\newtheorem{theorem}{Theorem}[section]

\newtheorem{lemma}[theorem]{Lemma}
\newtheorem{proposition}[theorem]{Proposition}

\theoremstyle{definition}

\newtheorem{remark}[theorem]{Remark}

\newcommand{\q}{\quad}

\newcommand{\e}{\varepsilon}

\renewcommand{\L}{\Lambda}

\def\theequation{\thesection.\arabic{equation}}

\newcommand\C{{\mathbb C}}
\newcommand\R{{\mathbb R}}

\newcommand\N{{\mathbb N}}
\renewcommand{\phi}{\varphi}

\newcommand{\mt}{\mu_T}
\newcommand{\mi}{\mu_I}
\newcommand{\mv}{\mu_V}
\allowdisplaybreaks

\title[Heterogeneous viral environment in a HIV spatial model]
{Heterogeneous viral environment\\ in a HIV spatial model}

\author[C.-M. Brauner, D. Jolly, L. Lorenzi, R. Thiebaut ]{}

\subjclass{Primary: 35K55; Secondary: 35B35, 92C50}
 \keywords{}

 \email{claude-michel.brauner@math.u-bordeaux1.fr}
 \email{danaelle.jolly@math.u-bordeaux1.fr}
 \email{luca.lorenzi@unipr.it}
 \email{rodolphe.thiebaut@isped.u-bordeaux2.fr}

\thanks{Work partially supported by the research project
ANR-BBSRC SysBio: Applied statistical and mathematical modelling of
peripheral T-Lymphocyte homeostasis}

\begin{document}
\maketitle

\centerline{\scshape Claude-Michel Brauner}
\medskip
{\footnotesize
 \centerline{Institut de Math\'ematiques de Bordeaux}
   \centerline{Universit\'e de Bordeaux,
33405 Talence cedex (France)}
} 

\medskip

\centerline{\scshape Danaelle Jolly}
\medskip
{\footnotesize
 \centerline{Institut de Math\'ematiques de Bordeaux}
   \centerline{Universit\'e de Bordeaux, 33405 Talence cedex (France)}
} %

\medskip

\centerline{\scshape Luca Lorenzi}
\medskip
{\footnotesize
 \centerline{Dipartimento di Matematica}
   \centerline{Universit\`a
di Parma, Viale G.P. Usberti 53/A, 43100 Parma (Italy)}
} %

\medskip
\centerline{\scshape Rodolphe Thiebaut}
\medskip
{\footnotesize
 \centerline{(M.D.)  Equipe Biostatistique de l'U897 INSERM
ISPED}
   \centerline{Universit\'e de Bordeaux,
33076 Bordeaux cedex (France)}
} %

\bigskip


\begin{abstract}
We consider the basic model of virus dynamics in the modeling of
Human Immunodeficiency Virus (HIV), in a $2D$ heterogenous
environment. It consists of two ODEs for the non-infected and
infected $CD_4^+$ $T$-lymphocytes, $T$ and $I$, and a parabolic PDE
for the virus $V$. We define a new parameter $\lambda_0$ as an
eigenvalue of some Sturm-Liouville problem, which takes the
heterogenous reproductive ratio into account. For $\lambda_0<0$ the
trivial non-infected solution is the only equilibrium. When
$\lambda_0>0$, the former becomes unstable whereas there is only one
positive infected equilibrium. Considering the model as a dynamical
system, we prove the existence of a universal attractor. Finally, in
the case of an alternating structure of viral sources, we define a
homogenized limiting environment. The latter justifies the classical
approach via ODE systems.
\end{abstract}

\section{Introduction}
The acute infection by the Human Immunodeficiency Virus (in short
HIV) is characterized by a huge depletion of the
$CD_4^+~T$-lymphocytes ($CD_4$) and a peak of the virus load
\cite{daar}. After few weeks, these two components reach a steady
state which characterizes the asymptomatic phase of the infection.
Before the availability of highly active antiretroviral therapy,
this later phase lasted after $10$ years in median with an
accelerated decrease of $CD_4$ and an increase of virus load. A
substantial number of nonlinear ODE systems have been suggested by
Perelson et al. (see \cite{perelson, perelson2}) to model the
complex dynamics of HIV-host interaction. For instance, such models
have been used to estimate the infected cell half-life and the viral
clearance during antiretroviral therapy \cite{ho,perelson0,wei}, or
to understand the dynamics during acute infection \cite{phillips}.

The common basic model of viral dynamics \cite{bonhoeffer} includes
three variables: $T$, the non-infected $CD_4$, $I$, the infected
$CD_4$, and $V$, the free virus:
\setcounter{equation}{1}
\begin{align}
&T_t=\alpha - \gamma V T - \mu_T T,\tag{\theequation a}
\label{ode1}
\\
&I_t= \gamma V T - \mu_I I,\tag{\theequation b}\label{ode2}
\\
&V_t=N \mi I- \mu_V V.\tag{\theequation c}\label{ode3}
\end{align}
This model, describing the interaction between the replicating virus
of HIV and host cells, is based on some simple hypotheses.
Non-infected target cells are produced by the thymus at a constant
rate $\alpha$ and die at a rate $\mt$. By contact with the free
virus particles (virions) they become infected at a rate
proportional to their abundance, $\gamma V T$ (see Fig.
\ref{contact}). These infected cells die at a rate $\mi$ and produce
free viruses during their life-time at a rate $N$. Free particles
are removed at a rate called the clearance $\mv$. All these
parameters are generally positive constants. This simple model study
has led to interesting results (see \cite{nowak2001,wangsong}) and
suggested a treatment strategy (see \cite{bonhoeffer}). \vskip
-5truemm
\begin{figure}[htp]
\begin{center}
\epsfig{file=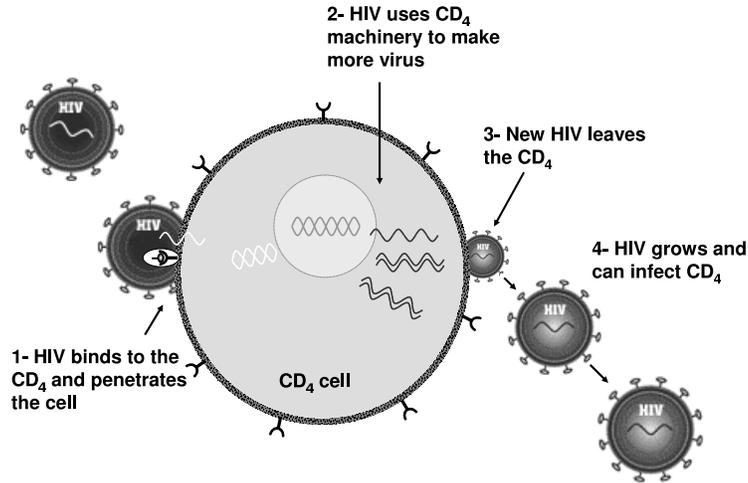,scale=0.4,angle=-90}
\\[6mm]
\caption{The HIV life cycle}
\end{center}
\label{contact}
\end{figure}%

It is easily seen that the system has two equilibria:
\begin{enumerate}[(i)]
\item
the {\it non-infected} steady state
\begin{eqnarray*}
T_u = \frac{\alpha}{\mt}, \quad I_u = 0, \quad V_u = 0,
\end{eqnarray*}
which corresponds to a non-negative equilibrium in case of no
infection;
\item
the {\it infected} steady state
\begin{equation}
T_i = \frac{\mv}{\gamma N}, \quad I_i = \frac{\alpha \gamma N - \mt
\mv}{\gamma N }, \quad V_i = \frac{\alpha \gamma N - \mt \mv}{\gamma
\mv}, \label{T-inst-const}
\end{equation}
also called {\it seropositivity} steady state, corresponding
to a positive equilibrium in case of infection.
\end{enumerate}

Some authors (see e.g., \cite{bonhoeffer,nowak2001}) have considered the
basic reproductive ratio $R_0$:
\begin{equation}
R_0= \frac{\gamma\alpha N}{\mt \mv},
\label{R0-def}
\end{equation}
a dimensionless parameter defined by epidemiologists as {\it the
average number of infected cells that derive from any one infected
cell in the beginning of the infection} \cite[p. 16]{nowak2001}.
Stability properties of the two steady states are usually studied
around this quantity: if $R_0<1$ the non-infected steady state is
stable, if $R_0>1$ the infected steady state has a biological
meaning and it is stable, while at $R_0=1$ both steady states
coincide. So, $R_0=1$ is a bifurcation point (see Fig.
\ref{bifur1}).
\begin{figure}[htp]
\begin{center}
\epsfig{file=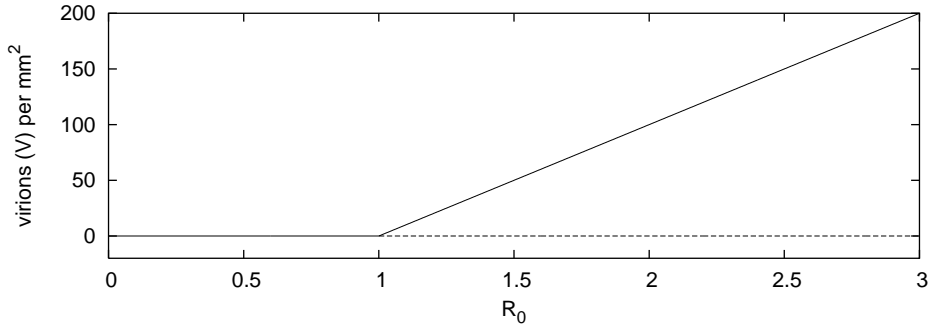,scale=1}
\end{center}
\caption{Typical bifurcation diagram, with $R_0$ varying between $0$
and $3$, as a function of $\alpha$. The solid line represents the
stable branch of $V$, the dashed line the unstable one. The other
parameters have the values:
$\gamma=0.001,~N=1000,~\mt=0.1,~\mi=0.5,~\mv=10$.} \label{bifur1}
\end{figure}%
\begin{figure}[htp]
\begin{center}
\epsfig{file=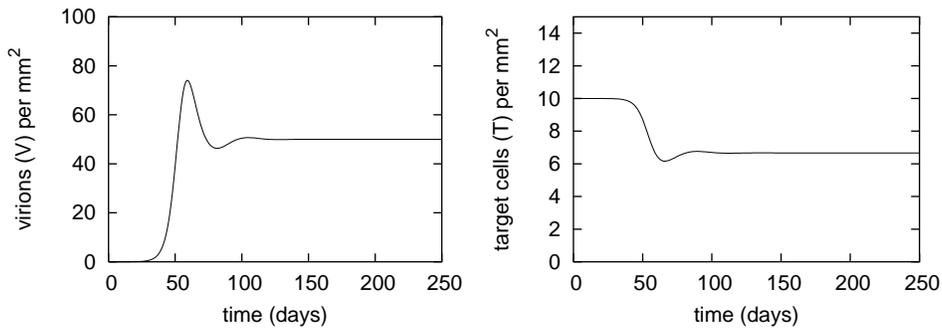,scale=1}
\end{center}
\caption{Profiles of virus (left) and target cells (right) in the
case of infection. Here $R_0=1.5$ and the other parameters have the
same values as in Fig. \ref{bifur1}.}
 \label{comparison}
\end{figure}%
Further models have been used involving other populations present in
the immune system (see \cite{nowak1996,nowak2001,callaway}).
However, these models assume that the populations $T,~I,~V$ are
homogeneous over the space for all time, which is a common, but not
a very realistic, assumption. Actually, the interaction between the
virus and the immune system (either as a target with $CD_4$ or as an
agent for controlling infection) is localized according to the type
of tissues \cite{brenchley} and also in a given tissue (e.g. lymph
nodes). To examine the effects of both diffusion and spatial
heterogeneity, Funk et al. \cite{funk} introduced a discrete model
based on \eqref{ode1}-\eqref{ode3}. These authors adopted a
two-dimensional square grid with $21 \times 21$ sites and assumed
that the virus can move to the eight nearest neighboring sites. They
pointed out that the presence of a spatial structure enhances
population stability with respect to non-spatial models. However,
our analysis does not confirm this observation (see Section
\ref{special} below).

Recently, Wang et al. \cite{Wang} generalized Funk et al.'s model.
They assumed that the hepatocytes can not move under normal
conditions and neglected their mobility, while viruses can move
freely and their motion follows a Fickian diffusion. They proposed
the following system of two ODEs coupled with a parabolic PDE for
the virus: \setcounter{equation}{4}
\begin{align}
&T_t=\alpha - \gamma V T - \mu_T T,\tag{\theequation a}\label{pde1}\\
&I_t= \gamma V T - \mu_I I,\tag{\theequation b}\label{pde2}\\
&V_t=N \mi I- \mu_V V + d_V \Delta V,\tag{\theequation
c}\label{pde3}
\end{align}
where $d_V$ is the diffusion coefficient. They assumed that the
domain is the whole real line and proved the existence of
traveling waves. 
Wang et al. \cite{kwang2008} introduced a delay to take into
account the time between infection of a target cell and the emission
of viral particles \cite{culshaw}. They considered \eqref{pde1}-\eqref{pde3} in
a one-dimensional interval with Neumann boundary conditions.

In the spirit of the above works, we intend to study System
\eqref{pde1}-\eqref{pde3} in a two-dimensional spatial domain $(0,\ell)\times (0,\ell)$
with periodic boundary conditions. There are two main situations:
\begin{enumerate}[\rm (i)]
\item
the environment is {\it homogeneous} and, hence, all the parameters in
\eqref{pde1}-\eqref{pde3} are constant. Therefore, the system with
diffusion has the same equilibria as System \eqref{ode1}-\eqref{ode3};
\item
the environment is {\it heterogeneous}, therefore certain parameters become positive
functions of the space variable. Then, the virus is spatially structured.
\end{enumerate}

For simplicity, we assume throughout the paper that only the rate
$\alpha$ varies while the other parameters are fixed positive
constants. In fact, it is biologically plausible to assume that the
arrival of new $CD_4$ may vary according to local areas. More
precisely, $\alpha$ is piecewise continuous and periodic in each
variable with period $\ell$.
Then it is convenient to define the {\it heterogeneous reproductive ratio}:
\begin{equation}
R_0(x)=\frac{\gamma N \alpha(x)}{\mt \mv},\qquad\;\,x\in
(0,\ell)\times (0,\ell). \label{R0-def-1}
\end{equation}
The sites where $R_0(x)<1$ are called {\it sinks} while the sites where
$R_0(x)>1$ are called {\it sources} \cite{funk}.

The paper is organized as follows. In Section \ref{sec-steady} we
are interested in the stationary problem associated with
\eqref{pde1}-\eqref{pde3} and its non-negative equilibria. The virus
equilibrium verifies the elliptic semilinear equation
\begin{equation}
\label{semil-intro} d_V \Delta V - \mv V =  -\mt\mv R_0(x)\frac{V}{\gamma V +\mt},
\end{equation}
with periodic boundary conditions.
A first issue is to define a parameter which will play
the role of the bifurcation parameter $R_0$ in the case the latter is constant.
A candidate for this role is the largest eigenvalue $\lambda_0$ of the operator:
\begin{equation}
\label{eigen-intro}
d_V\Delta +\mv(R_0-1)Id,
\end{equation}
which is the linearization around $V\equiv 0$ of
\eqref{semil-intro}. For the reader's convenience, we recall some
basic facts about two-dimensional Sturm-Liouville eigenvalue
problems with periodic boundary conditions such as
\eqref{eigen-intro} and give some proofs in Appendix \ref{app-A}.

It is clear that, whenever $R_0$ is a constant, $\lambda_0= \mv(R_0-1)$. Therefore, we
distinguish two cases, depending upon the sign of $\lambda_0$:
\begin{enumerate}[\rm (i)]
\item $\lambda_0 \leq 0 $: the trivial non-infected solution $V_u\equiv 0$ is the only solution
of \eqref{semil-intro};
\item $\lambda_0>0$: \eqref{semil-intro} has exactly two non-negative solutions, namely
the trivial non-infected solution $V_u\equiv 0$ and the positive
infected solution $V_i$.
\end{enumerate}

Section \ref{sec-3} is devoted to the study of the evolution problem
\eqref{pde1}-\eqref{pde3}. In the case $\lambda_0<0$, we prove that
the trivial non-infected solution $(T_i,I_i,V_i)$ is asymptotically
stable. Then, we turn our attention to the biologically relevant
case $\lambda_0>0$. First, we prove the non-infected solution
becomes unstable. Second, we consider \eqref{pde1}-\eqref{pde3} as a
dynamical system and prove the existence of an universal (or
maximal) attractor. Since the system is only partly dissipative, we
use a result of Marion \cite{marion}. The following Section
\ref{special} is devoted to some special cases where the positive
infected solution $V_i$ is stable. Particular attention is paid to
the case when $R_0$ is a constant: in this case discrete Fourier
transform can be applied.

We point out that our proof can be extended to further models in HIV
literature. It is not difficult to take a logistic term into account
in the $T$ equation \cite{perelson2}, although the steady equation
\eqref{semil-intro} will be more involved. Adding such a term, Hopf
bifurcations have been observed numerically in ODE systems (see
\cite{perelson0}). Therefore, proving the stability of the infected
solution may be, in general, challenging.

In the last section (Section \ref{sec-5}), we consider the case when
a heterogeneous environment is formed of sinks and sources
alternating very rapidly, with a heterogeneous reproductive ratio
$R_0(\frac{x}{\e})$. We determine the homogenized limiting medium as
$\e \to 0$. It is fully characterized by a constant reproductive
ratio, the mean value of $R_0$. Therefore, the classical approach of
HIV dynamics via ODEs in a homogenous environment can be {\em a
posteriori} justified in this respect.
\subsection*{Notation}
Throughout this paper, for any $\ell>0$, we denote by $L^2$ the
usual space of functions $f:(0,\ell)^2\to\mathbb R$ such that $f^2$
is integrable. The square $(0,\ell)^2$ will be simply denoted by
$\Omega_{\ell}$. By $H^{k}$ we denote the Sobolev space of order
$k$, i.e., the subset of $L^2$ of all the functions whose
distributional derivatives up to $k$-th order are in $L^2$. Both
$L^2$ and $H^2$ are endowed with their Euclidean norm. Finally, by
$H^k_{\sharp}$ we denote the closure in $H^k$ of the space
$C^m_{\sharp}$ of all $m$-th continuously differentiable functions
$f:\R^2\to\R$ which are periodic with period $\ell$ in each
variable. The space $H^k_{\sharp}$ is endowed with the norm of
$H^k$.
Finally, we denote by $Id$ the identity operator.

\section{A semilinear equation for the virus steady states: existence
and uniqueness of the equilibria} \label{sec-steady}\setcounter{equation}{1}

We start from the system for the virus dynamics:
\begin{align}
&T_t=\alpha - \gamma V T - \mu_T T,\tag{\theequation a}\label{pde1-bis}\\
&I_t= \gamma V T - \mu_I I,\tag{\theequation b}\label{pde2-bis}
\\
&V_t=N \mi I- \mu_V V + d_V \Delta V, \tag{\theequation
c}\label{pde3-bis}
\end{align}
set in $\Omega_{\ell}$. Periodic boundary conditions for $T, I$ and
$V$ are prescribed.

We are interested in the existence of steady state solutions to the
equations \eqref{pde1-bis}-\eqref{pde3-bis} which belong to the
space $L^2\times L^2\times H^2_{\sharp}$. Clearly, any steady state
solution to Problem \eqref{pde1-bis}-\eqref{pde3-bis} is a solution
to the following stationary system: \setcounter{equation}{2}
\begin{align}
&\alpha - \gamma V T - \mu_T T =0,\tag{\theequation
a}\label{ss1}\\
&\gamma V T - \mu_I I =0,\tag{\theequation
b}\label{ss2}\\
& N \mi I- \mu_V V + d_V \Delta V= 0.\tag{\theequation c}\label{ss3}
\end{align}
 From a biological point of view, only non-negative solutions to
\eqref{ss1}-\eqref{ss3} have a meaning. Hence, we limit ourselves to
proving the existence of this kind of steady state solutions.

System \eqref{ss1}-\eqref{ss3} can be reduced to a single scalar
equation for the unknown $V$. Actually, it is not difficult to infer
from \eqref{ss1}, \eqref{ss2} that
\begin{eqnarray*}
T=\frac{\alpha}{\gamma V +\mt}, \quad I=\frac{\gamma\alpha V}{\mi(\gamma V +\mt)}.
\end{eqnarray*}
Hence, the function $V$ turns out to solve the equation
\begin{equation}
\label{semil} d_V \Delta V - \mv V = - \frac{\gamma\alpha N
V}{\gamma V +\mt}= -\mt\mv R_0\frac{V}{\gamma V +\mt},
\end{equation}
associated with periodic boundary conditions.

\subsection{Existence and uniqueness of non-negative equilibria}
\label{sect-3} In this subsection we will provide a thorough study
of the equation \eqref{semil}. As it has been already stressed, we
are interested in non-negative solutions only.

Clearly, equation \eqref{semil} always admits the trivial
non-infected solution $V\equiv 0$ and, hence, Problem
\eqref{pde1-bis}-\eqref{pde3-bis} admits
\begin{equation}
T_u(x)=\frac{\alpha(x)}{\mt},\qquad I_u(x)=0,\qquad
V_u(x)=0,\qquad\;\,x\in \Omega_{\ell}, \label{non-infected-1}
\end{equation}
as a (trivial) steady state solution. We will call the triplet
$(T_u,I_u,V_u)$ the non-infected solution.

We are interested in studying the uniqueness of the non-infected
solution in the class of all the non-negative steady state solutions
to Problem \eqref{pde1-bis}-\eqref{pde3-bis}. Of course, in the case
when uniqueness does not hold (a situation which can actually occur,
look for instance at the case when $R_0>1$ and $\alpha$ is constant,
discussed in the introduction) we want to characterize all
biological relevant steady state solutions to Problem
\eqref{pde1-bis}-\eqref{pde3-bis}.

For this purpose, we need to recall the following results about Sturm-Liouville
eigenvalue problems in dimension two with periodic boundary conditions.

\begin{theorem}
\label{thm-2.1}\label{sturm-eigen} Let $d$ and $\mu$ be,
respectively, a positive constant and a bounded measurable function.
Further, let ${\mathscr A}:H^2_{\sharp}\to L^2$ be the operator
defined by ${\mathscr A}u=d\Delta u-\mu u$ for any $u\in
H^2_{\sharp}$. Then, the spectrum of ${\mathscr A}$ consists of
eigenvalues only. Moreover, its maximum eigenvalue $\lambda_{\max}$
is given by the following formula:
\begin{equation}
\label{variat} \lambda_{\max} = -\inf_{\psi \in H^1_{\sharp}, \psi
\not\equiv 0} \left\{\frac{d\int_{\Omega_{\ell}}|\nabla\psi|^2 dx +
\int_{\Omega_{\ell}}\mu\psi^2 dx} {\int_{\Omega_{\ell}}\psi^2 dx}
\right\}.
\end{equation}
Finally, the eigenspace corresponding to the eigenvalue
$\lambda_{\max}$ is one dimensional and contains functions which do
not change sign in $\overline{\Omega_{\ell}}$.
\end{theorem}
This is a rather classical result. Nevertheless, for the reader's
convenience, we give a proof in Appendix \ref{app-A}.

In view of Theorem \ref{thm-2.1}, we can define the constant
$\lambda_0$ to be the maximum eigenvalue of the operator
$\varphi\mapsto d_V\Delta \varphi+\mv(R_0-1)\varphi$, which is the
linearization around $V\equiv 0$ of operator $V\mapsto d_V \Delta V
- \mv V +\mt\mv R_0\frac{V}{\gamma V +\mt}$. According to
\eqref{variat},
\begin{equation}
\label{variat-1} -\lambda_0 = \inf_{\psi \in H^1_{\sharp}, \psi
\not\equiv 0} \left\{\frac{d_V\int_{\Omega_{\ell}}(\psi_x)^2 dx +
\mv\int_{\Omega_{\ell}}(1-R_0)\psi^2 dx} {\int_{\Omega_{\ell}}\psi^2
dx} \right\}. \end{equation}

As we are going to show, the uniqueness of the non-infected
steady state solution is related to the value of $\lambda_0$.

\begin{lemma}
\label{lemma-2.2} If $\lambda_0 \leq 0$, then the non-infected
solution $(T_u,I_u,V_u)$ $($see \eqref{non-infected-1}$)$ is the
only non-negative solution of \eqref{pde1-bis}-\eqref{pde3-bis}.
\end{lemma}
\begin{proof}
We argue by contradiction. Let us suppose that Problem
\eqref{ss1}-\eqref{ss3} admits another solution $(T^*,I^*,V^*)$
different from $(T_u,I_u,V_u)$. Then, the function $V^*\in
H^2_{\sharp}$ does not identically vanish in $\Omega_{\ell}$ and it
solves the equation \eqref{semil}. Multiplying both the sides of
this equation by $V^*$ and integrating by parts in $\Omega_{\ell}$,
we get:
\begin{eqnarray*}
d_V \int_{\Omega_{\ell}} |\nabla V^*|^2 dx + \mv\int_{\Omega_{\ell}}
\left(1 -
  \frac{\mt R_0}{\gamma V^*+\mt}\right)(V^*)^2 dx = 0,
  \end{eqnarray*}
or, equivalently,
\begin{align*}
d_V \int_{\Omega_{\ell}}& |\nabla V^*|^2 dx + \mv
\int_{\Omega_{\ell}}
\left(1-R_0\right)(V^*)^2 dx\\
&+ \mv\int_{\Omega_{\ell}} R_0\left(1 -
  \frac{\mt}{\gamma V^*+\mt}\right)(V^*)^2 dx=0.
\end{align*}
Since $V^*$ does not identically vanish in $\Omega_{\ell}$, the last
integral term is positive, implying that
\begin{eqnarray*}
d_V \int_{\Omega_{\ell}} |\nabla V^*|^2 dx + \mv
\int_{\Omega_{\ell}} \left(1 - R_0\right)(V^*)^2 dx < 0.
\end{eqnarray*}
Hence, the infimum in \eqref{variat} is negative which contradicts
our assumption $-\lambda_0 \geq 0$.
\end{proof}

\begin{remark}
\label{rem-2.1} {\rm From formula \eqref{variat-1} it is immediate
to check that, when the maximum of $R_0$ in $\Omega_{\ell}$ is less
than or equal to $1$, the constant $\lambda_0$ is non-positive.
Hence, in this situation the non-infected solution is the only
relevant steady state solution to Problem
\eqref{pde1-bis}-\eqref{pde3-bis} in complete agreement with the
case when $R_0$ is constant (see the Introduction).}
\end{remark}

The result in Lemma \ref{lemma-2.2} is very sharp as the following
theorem shows.

\begin{theorem}
\label{thm-2.4} Suppose that $\lambda_0>0$. Then, there exists a
steady state solution $(T_i,I_i,V_i)$ to Problem
\eqref{pde1-bis}-\eqref{pde3-bis} whose components are all positive
in $\overline{\Omega_{\ell}}$, with
\begin{eqnarray}
T_i=\frac{\alpha}{\gamma V_i +\mt}, \quad I_i=\frac{\gamma\alpha V_i}{\mi(\gamma V_i +\mt)}.
\label{infected-1}
\end{eqnarray}
Moreover, $(T_u,I_u,V_u)$ and
$(T_i,I_i,V_i)$ are the only steady state solutions whose components
are non-negative in $\Omega_{\ell}$.
\end{theorem}

\begin{proof}
It is clear, that we can limit ourselves to dealing with the
equation \eqref{semil}. Being rather long, we split the proof into
two steps.

{\em Step 1: $($existence$)$}. To prove the existence of a positive
solution to the equation \eqref{semil} in $H^2_{\sharp}$, we use the
classical method of upper and lower solutions. To  simplify the
notation, we denote by ${\mathscr R}$ the sup-norm of the function
$R_0$. We look for an upper solution $\overline{v}_0$ of \eqref{semil} as
a constant $C>0$. It is immediate to check that the best choice
of
$\overline{v}_0$ is
\begin{eqnarray*}
\overline{v}_0(x) \equiv \frac{\mt({\mathscr
R}-1)}{\gamma},\qquad\;\,x\in\overline{\Omega_{\ell}}.
\end{eqnarray*}
Note that, by Remark \ref{rem-2.1}, ${\mathscr R}$ is strictly
greater than $1$.

To determine a lower solution, in the spirit of \cite[Chapt. 13,
Sec. 3]{lions}, we take as a candidate to be a lower solution the
function $\underline{v}_0=c \varphi_0$ with $c>0$ to be fixed. Here,
$\varphi_0$ is the unique (positive) solution to the equation
$d_V\Delta \varphi_0+\mu_V(R_0-1)\varphi_0=\lambda_0\varphi_0$ which
satisfies $\sup_{x\in \Omega_{\ell}}\varphi_0(x)=1$.

If we plug $c\varphi_0$ in \eqref{semil}, we get
\begin{align*}
&d_V \Delta (c\varphi_0) - \mv c\varphi_0+\mt\mv R_0\frac{c\varphi_0}{\gamma c\varphi_0+\mt}\\
=\, &\lambda_0 c \varphi_0 + \mt\mv R_0\frac{c\varphi_0}{\gamma c\varphi_0+\mu_T}-\mv R_0c\varphi_0\\
=\, &c\varphi_0 \left(\lambda_0 - \gamma\mv R_0\frac{ c\varphi_0}{\gamma c\varphi_0+\mt}\right)\\
\geq\, &c\varphi_0 \left(\lambda_0 -\frac{\gamma c\mv{\mathscr
R}}{\mu_T}\right).
\end{align*}
Since $\lambda_0 >0$ is fixed, the last side of the previous chain
of inequalities is non-negative as soon as $\lambda_0\mu_T -\gamma
c\mv{\mathscr R}\geq 0$. Hence, if we fix
\begin{eqnarray*}
c= \min \left\{\frac{\lambda_0}{\gamma\mu_V{\mathscr R}},
\frac{\mt({\mathscr R} -1)}{\gamma} \right\},
\end{eqnarray*}
the function $\underline{v}_0=c\varphi_0$ turns out to be a positive
lower solution to the equation \eqref{semil} and it satisfies
$\underline{v}_0\le\overline{v}_0$ in $\Omega_{\ell}$.

Hence, the classical method of upper and lower solutions provides us
with a positive solution to the equation \eqref{semil}. It is enough
to define the sequence $(\overline{v}_n)$ by recurrence in following
way: for any fixed $n\in\mathbb N$, $\overline{v}_n$ is the unique
solution in $H^2_{\sharp}$ of the equation
\begin{eqnarray*}
d_V\Delta\overline{v}_n-\mv \overline{v}_n +\mt\mv
R_0\frac{\overline{v}_{n-1}}{\gamma \overline{v}_{n-1} +\mt}=0.
\end{eqnarray*}
Since the function $t\mapsto h(t):=\frac{t}{\gamma t +\mt}$ is
increasing in $[0,+\infty)$, the maximum principle (see Proposition
\ref{max-princ}) shows that the sequence $(\overline{v}_n)$ is
pointwise non-increasing. Moreover,
$\underline{v}_0\le\overline{v}_n\le\overline{v}_0$ for any
$n\in\N$. Hence, the sequence $(\overline{v}_n)$ is bounded.
Moreover, by very general results for the heat equation, there exist
positive constants $C_1$ and $C_2$, independent of $n$, such that
\begin{eqnarray*}
\|\overline{v}_n\|_{H^2_{\sharp}}\le C_1\|R_0h(\overline{v}_{n-1})\|_{L^2} \le
C_2\|\overline{v}_{n-1}\|_{L^2},
\end{eqnarray*}
for any $n\in\N$. Thus, the sequence $(\overline{v}_n)$ is bounded
in $H^2_{\sharp}$ as well. Since it converges pointwise in
$\overline \Omega_{\ell}$, we can now infer that $v_n$ converges
strongly in $H^1_{\sharp}$ and weakly in $H^2_{\sharp}$ to a
function $v\in H^2_{\sharp}$ which, of course, turns out to be a
solution to the equation \eqref{semil}. For further details on the
method of lower and upper solutions, we refer the reader, e.g., to
the monograph \cite{pao}.

\emph{Step 2: $($uniqueness$)$}. To prove the uniqueness of the
nontrivial non-negative solution to the equation \eqref{semil}, we
adapt to our situation a method due to H.B. Keller \cite{keller}.

Let us suppose that $u$ is another nontrivial non-negative solution
to \eqref{semil}. Then, the function $w:=v-u$ belongs to
$H^2_{\sharp}$ and solves the equation
\begin{equation}
d_V \Delta w - \mv w+\mt^2\mv R_0q(u,v)w=0, \label{eq-w}
\end{equation}
where
\begin{eqnarray*}
q(x,y)=\frac{1}{(\gamma x+\mt)(\gamma y+\mt)},\qquad\;\,x,y\ge 0.
\end{eqnarray*}
Let us denote by $\lambda_{\max}(u,v)$ and $\lambda_{\max}(u,0)$,
the maximum eigenvalues in $L^2$ of the operators
\begin{eqnarray*}
d_V\Delta+\mt^2\mv R_0q(u,v)Id
\end{eqnarray*}
and
\begin{eqnarray*}
d_V \Delta +\mt^2\mv R_0q(u,0)Id,
\end{eqnarray*}
respectively. By
Theorem \ref{thm-2.1}, they are given by the formula
\begin{eqnarray*}
\lambda_{\max}(u,jv)=-\inf_{\psi \in H^1_{\sharp}, \psi \not\equiv
0} \left\{\frac{d_V\int_{\Omega_{\ell}}|\nabla\psi|^2 dx -
\mt^2\mv\int_{\Omega_{\ell}}R_0q(u,jv)\psi^2 dx}
{\int_{\Omega_{\ell}}\psi^2 dx} \right\},
\end{eqnarray*}
for $j=0,1$. Since $v$ is non-negative and it does not identically
vanish, $q(u,v)\le q(u,0)$ and there exists an open subset of
$\Omega_{\ell}$ where $q(u,v)<q(u,0)$. Hence,
$\lambda_{\max}(u,v)<\lambda_{\max}(u,0)$. Clearly, since $w$
satisfies \eqref{eq-w} and it does not identically vanish in
$\overline{\Omega_{\ell}}$, then
$\mv\le\lambda_{\max}(u,v)<\lambda_{\max}(u,0)$.

Let us now rewrite the equation satisfied by $u$ in the following way:
\begin{eqnarray*}
d_V\Delta u-\lambda_{\max}(u,0)u+\mt^2\mv
R_0q(u,0)u=(\mv-\lambda_{\max}(u,0))u:=\varphi.
\end{eqnarray*}
Fredholm alternative implies that $\varphi$ should be orthogonal to
$\zeta$, where by $\zeta$ we have denoted the function which spans
the eigenspace associated with the eigenvalue $\lambda_{\max}(u,0)$.
As it has been already remarked (see Theorem \ref{thm-2.1}), the
function $\zeta$ does not change sign in $\overline{\Omega_{\ell}}$.
Similarly, since $\mv<\lambda_{\max}(u,0)$, $\varphi$ is
non-positive and it does not identically vanish since $u$ does not.
Hence, the function $\varphi$ cannot be orthogonal to $\zeta$ and
this leads us to a contradiction. The proof is now complete.
\end{proof}
\subsection{Numerical illustration (steady state)}\label{numerics1}
In accordance with Funk et al. \cite{funk},  the domain
$\Omega_{\ell}$ is a discrete square grid with $n \times n$ sites of equal
dimension $\ell/n \times \ell/n$. We assume in this numerical part
that all parameters vary randomly from site to site in such a way
that $0.1 \leq R_0(x) \leq 5.0$. We deal with two cases:
\begin{enumerate} [(i)]
\item  In the first case, see Fig. \ref{eqntr} (left),
the distribution of $R_0$ is as in Fig. \ref{R012}a. The sources
represent only $26\%$ of the sites. We compute $\lambda_0 = -1.70$.
Solving numerically the equation \eqref{semil}, we find the
non-infected solution $V_u\equiv 0$. We also represent $T_u$
according to Formula \eqref{non-infected-1}. \vskip -3truemm
\begin{figure}[htp]
\begin{center}
\begin{tabular}{cc}
\epsfig{figure=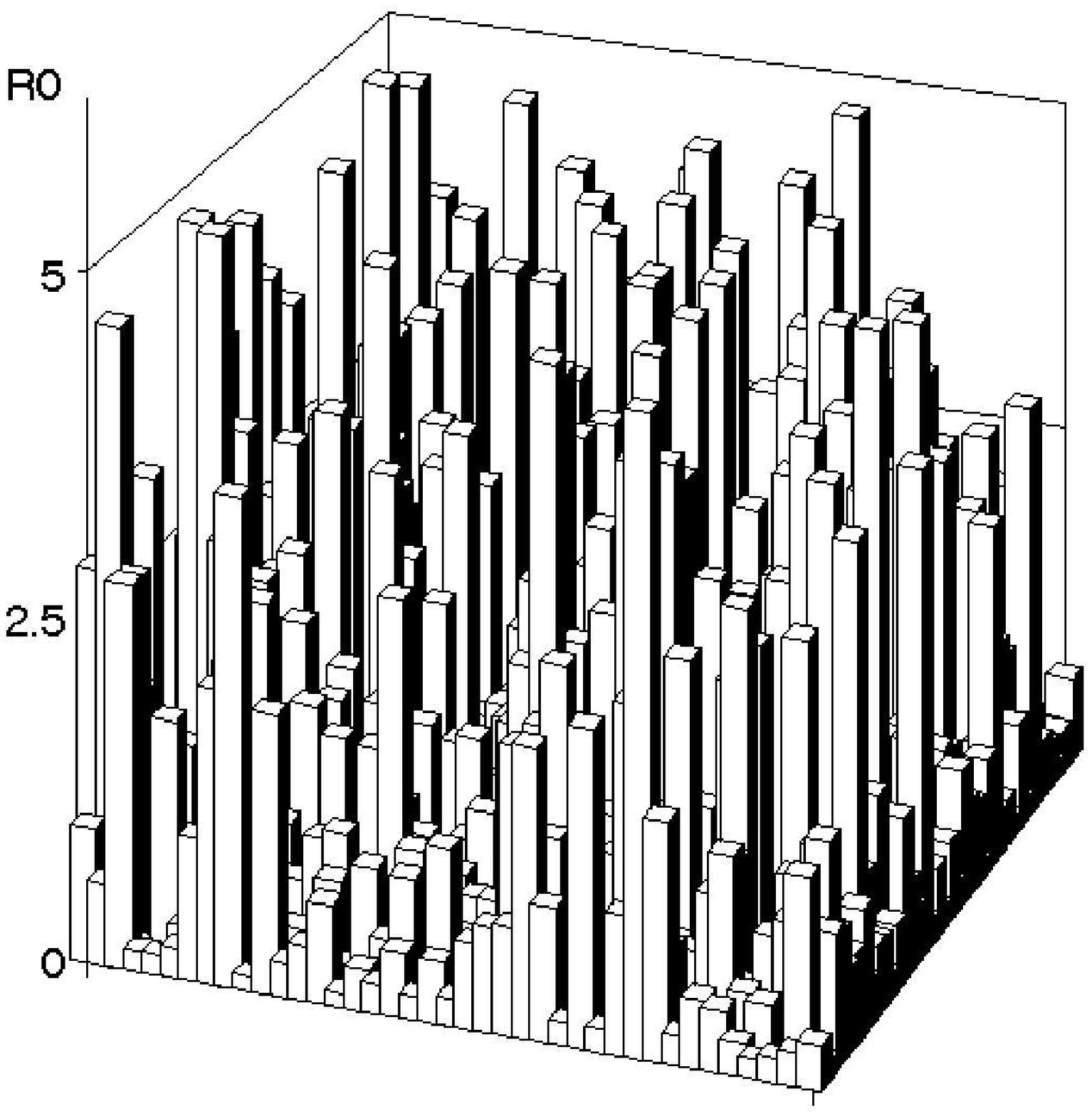,width=7cm} &\!\!\!\!\!\!\!\! \epsfig{figure=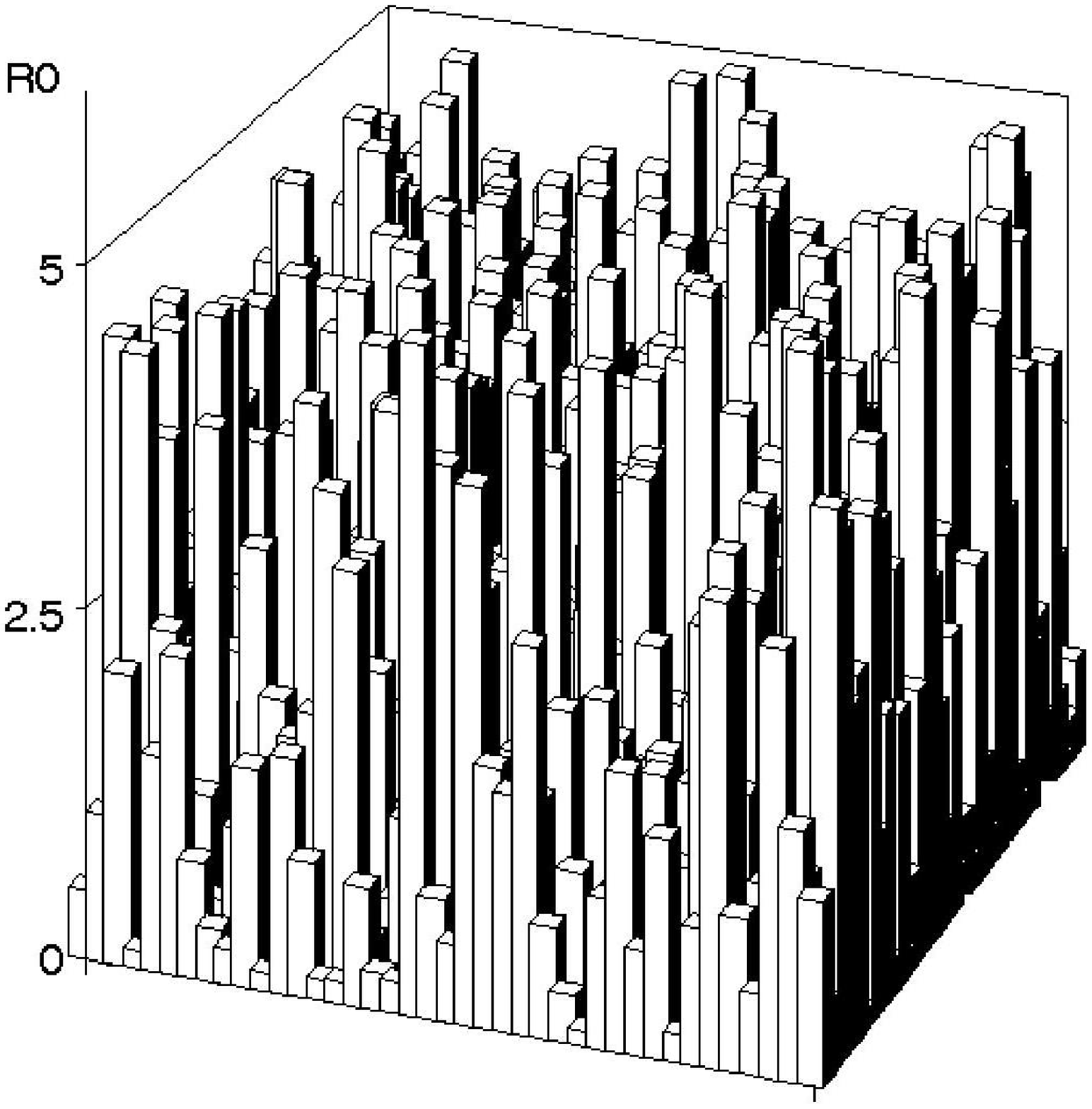,width=7cm}\\
(a) & (b)\\
\end{tabular}
\end{center}
 \caption{Two distributions of $R_0$ on a $40 \times 40$ grid. (a): $26\%$
 of the sites are sources; (b): $50\%$ of the sites are sources. \label{R012}}
\end{figure}

\item In the second case, see Fig. \ref{eqntr} (right), the distribution of $R_0$ is
as in Fig. \ref{R012}b. Now the sources represent half of the sites.
The eigenvalue $\lambda_0 = 4.30$ is positive. Numerically, we
observe the positive infected solution $V_i$ of \eqref{semil}. Note that
$V_i$ is smoothly structured in space although $R_0$ is not.
We also represent $T_i$ according to Formula \eqref{infected-1}.
\begin{figure}[ntp]
\begin{center}
\epsfig{file=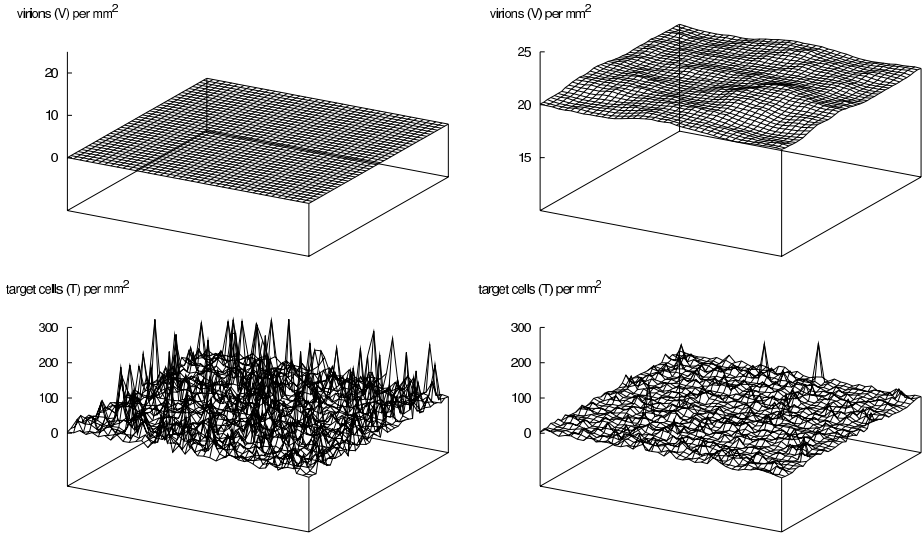,scale=0.55}
\end{center}
\caption{Densities of virus $V$ (top) and target cells $T$ (bottom).
Left: $R_0$ as in Fig. \ref{R012}a, $\lambda_0<0$ (no infection).
Right: $R_0$ as in Fig. \ref{R012}b, $\lambda_0>0$ (infection). Here
$d_V=1,~\ell=1,~N=40$.} \label{eqntr}
\end{figure}%
\end{enumerate}
\section{Study of the dynamical system}
\label{sec-3}
\setcounter{equation}{1} We recall the evolution
problem for the virus dynamics:
\begin{align}
&T_t=\alpha - \gamma V T - \mu_T T,\tag{\theequation a}\label{pde1-ter}\\
&I_t= \gamma V T - \mu_I I,\tag{\theequation b}\label{pde2-ter}
\\
&V_t=N \mi I- \mu_V V + d_V \Delta V, \tag{\theequation
c}\label{pde3-ter}
\end{align}
set in $\Omega_{\ell}$ with periodic boundary conditions. We
consider \eqref{pde1-ter}-\eqref{pde3-ter} as a dynamical system
${\mathscr S}(t)$, which has two equilibria: the non-infected
trivial solution $(T_u,I_u,V_u)$ and the infected, positive solution
$(T_i,I_i,V_i)$, the latter for $\lambda_0>0$ only. At first, we
prove that the non-infected solution is stable for $\lambda_0<0$ and
unstable for $\lambda_0>0$. By stable we mean {\em asymptotically}
stable. For $\lambda_0>0$ the instability of $(T_u,I_u,V_u)$ does
not usually imply the stability of $(T_i,I_i,V_i)$. Our aim is to
prove the existence of a universal (or maximal) attractor which
attracts all the orbits (see e.g., \cite{temam}). Since System
\eqref{pde1-ter}-\eqref{pde3-ter} is only partly dissipative, we
will use a result of Marion \cite{marion}. Some special cases where
the stability of the infected solution is granted will be discussed
afterwards.

Let us introduce the following notations:
${\mathscr D}$ is the domain in $\R^3$ defined by:
\begin{eqnarray*}
{\mathscr D}=\{(x,y,z)\in\R^3: 0\leq x+y\leq M_1,\; 0\leq z\leq
M_2\},
\end{eqnarray*}
where $M_1$, $M_2$ are positive constants which will be fixed
throughout the proof of the next theorem. We also set for ${\bf
u}:=(u_1,u_2,u_3)$:
\begin{eqnarray*}
{\bf H}= \{{\bf u} \in (L^2)^3:\ {\bf u}(x) \in {\mathscr D}
\;\,\mbox{for a.e.}\; x \in \Omega_{\ell}\},\quad {\bf V}= L^2
\times L^2  \times H_{\sharp}^1.
\end{eqnarray*}
Our main result is the following theorem.
\begin{theorem}\label{stab-lambda_0-positive}
The following properties are met:
\begin{enumerate}[\rm (i)]
\item
if $\lambda_0<0$, the trivial non-infected solution to Problem
\eqref{pde1-ter}-\eqref{pde3-ter} is stable;
\item
if $\lambda_0>0$, the trivial non-infected solution to Problem
\eqref{pde1-ter}-\eqref{pde3-ter} is unstable;
\item
if $\lambda_0>0$ and $\alpha\in W^{1,\infty}(\Omega_{\ell})$, the
dynamical system ${\mathscr S}(t)$ associated with
\eqref{pde1-ter}-\eqref{pde3-ter} possesses a universal attractor
that is connected in {\bf H}.
\end{enumerate}
\end{theorem}

\subsection{Proof of (i)}
We begin the proof observing that the linearization (around
$(T_u,I_u,V_u)$) of Problem \eqref{pde1-ter}-\eqref{pde3-ter} is
associated with the linear operator ${\mathscr L}_u$ defined by
\begin{eqnarray*}
{\mathscr L}_u=
\begin{pmatrix}
- \mu_T Id & 0 & -\displaystyle\frac{\gamma\alpha}{\mu_T}Id \\[3.5mm]
0 &- \mu_I Id & \displaystyle\frac{\gamma\alpha}{\mu_T}Id \\[3.5mm]
0 & N \mi Id & d_V\Delta-\mv Id
\end{pmatrix}.
\end{eqnarray*}
Its realization $L_u$ in $(L^2)^3$ with domain $D(L_u)=L^2\times
L^2\times H^2_{\sharp}$ generates an analytic strongly continuous
semigroup. Indeed, $L_u$ is a bounded perturbation of the diagonal
operator
\begin{eqnarray*}
\begin{pmatrix}
- \mu_T Id & 0 & 0 \\[3.5mm]
0 &- \mu_I Id & 0 \\[3.5mm]
0 & 0 & d_V\Delta-\mv Id
\end{pmatrix},
\end{eqnarray*}
defined in $L^2\times L^2\times H^2_{\sharp}$, which is clearly
sectorial since all its entries are. Hence, we can apply \cite[Prop.
2.4.1(i)]{lunardi} and conclude that $L_u$ is sectorial. Since
$H^2_{\sharp}$ is dense in $L^2$, the associated analytic semigroup
is strongly continuous.

Let us prove that all the elements of the spectrum of $L_u$ have
negative real part. In view of the linearized stability principle
(see e.g., \cite[Chapt. 5, Cor. 5.1.6]{henry}) this will imply that
the trivial non-infected solution to Problem
\eqref{pde1-ter}-\eqref{pde3-ter} is stable.

To study the spectrum of the operator $L_u$, we fix ${\bf
f}=(f_1,f_2,f_3)\in (L^2)^3$ and consider the resolvent system
\setcounter{equation}{2}
\begin{align}
&(\lambda+\mu_T)\varphi_1+\frac{\alpha\gamma}{\mu_T}\varphi_3=f_1,\tag{\theequation
a}
\label{spectr-1}\\
&(\lambda+\mu_I)\varphi_2-\frac{\alpha\gamma}{\mu_T}\varphi_3=f_2,\tag{\theequation
b}
\label{spectr-2}\\
&(\lambda+\mu_V)\varphi_3-d_V\Delta
\varphi_3-N\mu_I\varphi_2=f_3,\tag{\theequation c} \label{spectr-3}
\end{align}
where we look for a triplet of functions $\varphi_1,\varphi_2\in
L^2$ and $\varphi_3\in H^2_{\sharp}$. Suppose that $\lambda$ differs
from both $-\mu_I$ and $-\mu_T$ (which belong to the essential
spectrum). Then, we can use equations \eqref{spectr-1} and
\eqref{spectr-2} to make $\varphi_1$ and $\varphi_2$ explicit in
terms of $\varphi_3$. In particular, replacing the expression of
$\varphi_2$ in terms of $\varphi_3$ into \eqref{spectr-3} and using
the very definition of the function $R_0$ (see \eqref{R0-def-1}), we
can transform Problem \eqref{spectr-1}-\eqref{spectr-3} into the
equivalent equation for $\varphi_3$ only:
\begin{equation}
(\lambda+\mu_V)\varphi_3-d_V\Delta
\varphi_3-\frac{\mu_I\mu_V}{\lambda+\mu_I}R_0\varphi_3=\frac{N\mu_I}{\lambda+\mu_I}f_2
+f_3. \label{eq-v3-0}
\end{equation}
Adding and subtracting $-\mu_VR_0$ from the left-hand side of
\eqref{eq-v3-0}, we can rewrite the equation \eqref{eq-v3-0} into
the equivalent form:
\begin{equation}
(\lambda+\mu_V)\varphi_3-d_V\Delta
\varphi_3-\mu_VR_0\varphi_3+\mu_VR_0\frac{\lambda}{\lambda+\mu_I}\varphi_3=\frac{N\mu_I}{\lambda+\mu_I}f_2
+f_3. \label{eq-v3}
\end{equation}
Note that, for any $\lambda\in\C$, the operator $A_{\lambda}$
defined by the left-hand side of \eqref{eq-v3} has compact
resolvent. Hence, its spectrum consists of eigenvalues only. We are
going to prove that, for $\lambda\in\C$ with non-negative real part,
$0$ is not an eigenvalue of $A_{\lambda}$. For this purpose, we
observe that $-\lambda_0$ can be equivalently characterized as the
infimum of the ratio
\begin{eqnarray*}
\frac{d_V\int_{\Omega_{\ell}}|\nabla\psi|^2 dx +
\mv\int_{\Omega_{\ell}}(1-R_0)|\psi|^2 dx}
{\int_{\Omega_{\ell}}|\psi|^2dx}
\end{eqnarray*}
when $\psi$ runs in the set of all the complex-valued functions
$\psi=\psi_1+i\psi_2$, with $\psi_1,\psi_2\in H^1_{\sharp}$. This
shows, in particular, that
\begin{equation}
-\lambda_0\int_{\Omega_{\ell}}|\psi|^2dx\le
d_V\int_{\Omega_{\ell}}|\nabla\psi|^2dx
+\mu_V\int_{\Omega_{\ell}}(1-R_0)|\psi|^2dx, \label{variat-psi}
\end{equation}
for any function $\psi$ as above.

Let now $\varphi_3$ be a complex-valued solution to \eqref{eq-v3}.
Multiplying both the sides of such an equation by the conjugate of
$\varphi_3$ and integrating by parts, we easily see that
\begin{align}
(\lambda&+\mu_V)\int_{\Omega_{\ell}}|\varphi_3|^2dx+
d_V\int_{\Omega_{\ell}}|\nabla
\varphi_3|^2dx-\mu_V\int_{\Omega_{\ell}}R_0|\varphi_3|^2dx\notag\\
&+\frac{\mu_V\lambda}{\lambda+\mu_I}\int_{\Omega_{\ell}}R_0|\varphi_3|^2dx=0.
\label{eq-Re-l}
\end{align}
Taking the real part of both the sides of \eqref{eq-Re-l} and using
\eqref{variat-psi}, we obtain
\begin{eqnarray*}
({\rm
Re}\lambda+\mu_V)\int_{\Omega_{\ell}}|\varphi_3|^2dx-\lambda_0\int_{\Omega_{\ell}}|\varphi_3|^2dx
+\mu_V\frac{|\lambda|^2+\mu_I{\rm
Re}\lambda}{|\lambda+\mu_I|^2}\int_{\Omega_{\ell}}|\varphi_3|^2dx
\le 0.
\end{eqnarray*}
Since, by assumptions, $\lambda_0<0$ and $R_0$ is a positive-valued
function, the only solution to the previous inequality, when ${\rm
Re}\lambda\ge 0$, is the trivial function $\varphi_3\equiv 0$.
Hence, for these values of $\lambda$, $0$ is not an eigenvalue of
the operator $A_{\lambda}$, i.e., any $\lambda$ with non-negative
real part belongs to the resolvent set of the operator $L_u$.

\subsection{Proof of (ii)}
Again in view of the linearized stability principle, to prove the
instability of non-infected solution $(T_u,I_u,V_u)$ we can limit
ourselves to showing that $L_u$ admits a positive eigenvalue. Hence,
we are led to the study of Problem \eqref{spectr-1}-\eqref{spectr-3}
with $f_1\equiv f_2\equiv f_3\equiv 0$. Since the parameters
$\mu_T$, $\mu_I$ and $\mu_V$ are all positive and we are looking for
positive eigenvalues $\lambda$, we can limit ourselves, as in the
proof of (i), to studying the equation
\begin{eqnarray*}
d_V\Delta\varphi_3 +\mv\left (\frac{\mi R_0}{\lambda+\mi}-1\right )
\varphi_3= \lambda\varphi_3.
\end{eqnarray*}

For any $s\ge 0$, let us consider the operator $d_V\Delta +\mv\left
(\frac{\mi R_0}{s+\mi}-1\right )Id$ defined in $H^2_{\sharp}$. By
Theorem \ref{thm-2.1} its spectrum consists of eigenvalues only, and
the largest one is given  by the following formula:
\begin{eqnarray*}
-\lambda(s) = \inf_{\psi \in H^1_{\sharp}, \psi \not\equiv 0}
\left\{\frac{d_V\int_{\Omega_{\ell}}|\nabla\psi|^2 dx +\mv
\int_{\Omega_{\ell}}\left (1-\frac{\mi R_0}{s+\mi}\right )\psi^2 dx}
{\int_{\Omega_{\ell}}\psi^2 dx} \right\}.
\end{eqnarray*}
As it is immediately seen, $\lambda(0)=\lambda_0$. Moreover, since
the function $s\mapsto \frac{\mi R_0}{s+\mi}$ is continuous,
decreasing in $[0,+\infty)$ and it tends to $0$ as $s\to +\infty$,
the function $s\mapsto\lambda(s)$ is continuous, decreasing and
tends to
\begin{eqnarray*}
\lambda(+\infty) = -\inf_{\psi \in H^1_{\sharp}, \psi \not\equiv 0}
\left\{\frac{d_V\int_{\Omega_{\ell}}|\nabla\psi|^2 dx +\mv
\int_{\Omega_{\ell}}\psi^2 dx} {\int_{\Omega_{\ell}}\psi^2 dx}
\right\},
\end{eqnarray*}
as $s\to +\infty$, i.e., it converges to the largest eigenvalue of
the operator $d_V\Delta-\mu_V Id$, which clearly is $-\mu_V$. Now,
since $s\mapsto\lambda(s)$ is a decreasing continuous function
mapping $[0,+\infty)$ into $[-\mv,\mv^{-1}\lambda_0]$, it is
immediate to check that the fixed point equation $s=\lambda(s)$ has
a positive solution. Of course, this fixed point is the positive
eigenvalue of the operator $L_u$ we were looking for. This
accomplishes the proof.
\endproof

\subsection{Proof of (iii)}

Let us show that \cite[Thm. 5.1]{marion} applies. In this respect
the assumption $\alpha\in W^{1,\infty}(\Omega_{\ell})$ is enough. To
avoid conflict with notations, throughout the proof, $T$ denotes
time as usual, whereas the triplet $(T,I,V)$ is denoted by
$(u_1,u_2,u_3)$. We split the proof into several steps.

{\em Step 1.} Here, we prove that, for any ${\bf
u_0}:=(u_{0,1},u_{0,2},u_{0,3}) \in (L^2)^3$, the Cauchy problem
\begin{equation}
\left\{
\begin{array}{ll}
D_tu_1(t,\cdot)=-\mu_Tu_1(t,\cdot)-\gamma
u_1(t,\cdot)u_3(t,\cdot)+\alpha(\cdot),
& t>0,\\[2mm]
D_tu_2(t,\cdot)=-\mu_Iu_2(t,\cdot)+\gamma
u_1(t,\cdot)|u_3(t,\cdot)|,
& t>0,\\[2mm]
D_tu_3(t,\cdot)=\Delta u_3(t,\cdot)-\mu_Vu_3(t,\cdot)+N\mu_I
u_2(t,\cdot),\q & t>0,\\[2mm]
u_i(0,\cdot)=u_{0,i},\;\; i=1,2,3,
\end{array}
\right. \label{pb-u2-u3}
\end{equation}
Problem \eqref{pb-u2-u3} admits a unique classical solution defined
in some time domain $(0,T_*)$. Here, by classical solution, we mean
a vector valued function ${\bf u}$ such that $u_1,u_2\in
C^1((0,T_*);L^2)\cap C([0,T_*);L^2)$ and $u_3\in C([0,T_*);L^2)\cap
C^1((0,T_*);L^2) \cap C((0,T_*);H^2_{\sharp})$.

Problem \eqref{pb-u2-u3} is semilinear with a nonlinear term which
is a continuous function from $L^2\times L^2\times
D_{\Delta_2}(\beta,2)$ into $L^2\times L^2\times L^2$ for any $\beta
\in (1/2,1)$. Here, $D_{\Delta_2}(\beta,2)$ is the interpolation
space of order $(\beta,2)$ between $L^2$ and the domain of the
realization $\Delta_2$ of the Laplacian with periodic boundary
conditions in $L^2$ (i.e., $D(\Delta_2)=H^2_{\sharp}$). Hence,
$D_{\Delta_2}(\beta,2)=(L^2,H^2_{\sharp})_{\beta,2}$ and this latter
space coincides with $H^{2\beta}_{\sharp}$, which continuously
embeds into the space $C_{\sharp}$ of all continuous and periodic
(with period $\ell$ in each variable) functions (see e.g.,
\cite[Thm. 1.4.4.1]{grisvard}).

To prove the existence of a classical solution to Problem
\eqref{pb-u2-u3}, let us fix $\beta \in (1/2,1)$ and introduce, for
any $T>0$, the space $C_{\beta}(T)$ consisting of all functions
$v:(0,T)\to D_{\Delta_2}(\beta,\infty)$ such that
$\|v\|_{C_{\beta}(T)}:=\sup_{t\in
(0,T]}\|t^{\beta}v(t,\cdot)\|_{D_{\Delta_2}(\beta,\infty)}<+\infty$.
Clearly, $C_{\beta}(T)$ is a Banach space when endowed with the
above norm. Moreover, $D_{\Delta_2}(\beta,\infty)$ is continuously
embedded into $D_{\Delta_2}(\beta-\varepsilon,2)$ for any
$\varepsilon\in (0,\beta)$. We now fix $\varepsilon$ small enough
such that $\theta=\beta-\varepsilon>1/2$. From the above results, it
is immediate to infer that $C_{\beta}(T)$ is embedded into the set
of all continuous functions $f:(0,T]\times\overline
\Omega_{\ell}\to\R$ and there exists a positive constant $C_{1}$,
independent of $T$, such that
\begin{equation}
\sup_{t\in (0,T]}t^{\theta}\|f(t,\cdot)\|_{\infty}\le
C_1\|f\|_{C_{\beta}(T)}. \label{crucial-estim}
\end{equation}

Let us solve the Cauchy problem for $u_1$, taking $u_3$ as a
parameter. The (unique) solution to such a problem in
$C^1((0,T];L^2)\cap C([0,T];L^2)$ is the function $u_1$ defined by
\begin{align}
u_1(t,\cdot)=&\exp\left (-\mu_Tt-\gamma\int_0^tu_3(s,\cdot)ds\right
)u_{0,1}\notag\\
&+\alpha(\cdot)\int_0^t\exp\left
(-\mu_T(t-s)-\gamma\int_s^tu_3(r,\cdot)dr\right )ds,
\label{op-Lambda-1}
\end{align}
for any $t\in (0,T]$. If $u_{0,1}$ is bounded and continuous in
$\Omega_{\ell}$ this result is straightforward. In the general case,
we approximate $u_{0,1}\in L^2$ by a sequence of smooth functions
$u_{0,1}^{(n)}$. It is immediate to check that the function
$u_1^{(n)}$ defined by \eqref{op-Lambda-1}, with $u_{0,1}^{(n)}$
instead of $u_{0,1}$, converges to the function $u_1$ in
$\eqref{op-Lambda-1}$ in $C([0,T];L^2)$, by dominated convergence.
Similarly, $D_tu_1^{(n)}$ converges to $D_tu_1$ in
$C([\varepsilon,T-\varepsilon];L^2)$ for any $\varepsilon>0$. It
follows that the function $u_1$ in \eqref{op-Lambda-1} is a solution
to the Cauchy problem for $u_1$ also in the case when $u_{0,1}$ is
in $L^2$.

Let us now denote by $\Lambda_1$ the operator defined in
$C_{\beta}(T)$ by the right-hand side of \eqref{op-Lambda-1}. A very
easy computation shows that $\Lambda_1$ maps $C_{\beta}(T)$ into
$C^{1-\theta}([0,T];L^2)$. Moreover, \setcounter{equation}{10}
\begin{align}
&\|\Lambda_1(v)\|_{C([0,T];L^2)}\le \exp\left
(\frac{C_1}{1-\theta}\gamma T^{1-\theta}\|v\|_{C_{\beta}(T)}\right
)\left (\|u_{0,1}\|_{L^2}+T\|\alpha\|_{L^2}\right
),\tag{\theequation a}
\label{sup-norm-Lambda}\\[2mm]
&\|\Lambda_1(v)-\Lambda_1(w)\|_{C([0,T];L^2)}\notag\\
 \le &\exp\left
(\frac{C_1}{1-\theta}\gamma
T^{1-\theta}\max(\|v\|_{C_{\alpha}(T)},\|w\|_{C_{\beta}(T)}) \right
)\notag\\
&\qquad\qquad\qquad\times\left
(1+\frac{C_1}{(1-\theta)(2-\theta)}\gamma
T^{2-\theta}\|\beta\|_{L^2}\right
)\|w-v\|_{C_{\beta}(T)}.\tag{\theequation b} \label{sup-contr}
\end{align}

We now consider the equation for $u_2$. Replacing
$u_1=\Lambda_1(u_3)$ in the right-hand side of this equation and
using the same argument as above, we easily see that the (unique)
solution in $C([0,T];L^2)\cap C^1((0,T];L^2)$ is the function $u_2$
defined by
\begin{equation}
u_2(t,\cdot)=e^{-t\mu_I}u_{0,2}+\gamma\int_0^te^{-\mu_I(t-s)}|u_3(s,\cdot)|(\Lambda_1(u_3))(s,\cdot)ds,
\qquad\;\,t\in [0,T].
\label{operat-Lambda2}
\end{equation}

Let us denote by $\Lambda_2$ the operator defined in $C_{\beta}(T)$
by the right-hand side of \eqref{operat-Lambda2}. Taking
\eqref{crucial-estim}, \eqref{sup-norm-Lambda} and \eqref{sup-contr}
into account, one can easily show that
\setcounter{equation}{12}
\begin{align}
&\|\Lambda_2(v)\|_{C([0,T];L^2)}\le
\|u_{0,2}\|_{L^2}+\frac{C_1}{1-\theta}\gamma
T^{1-\theta}\|v\|_{C_{\beta}(T)}\left (\|u_{0,1}\|_{L^2}+T\|\alpha\|_{L^2}\right )\notag\\
&\qquad\qquad\qquad\qquad\qquad\qquad\qquad\qquad\times  \exp\left
(\frac{C_1}{1-\theta}\gamma T^{1-\theta}\|v\|_{C_{\beta}(T)}\right
),\tag{\theequation a}
\label{sup-norm-Lambda-1}\\[2mm]
&\|\Lambda_2(v)\|_{C^{1-\theta}([0,T];L^2)}\le \left (
\frac{1}{\mu_I}T^{\theta}+1\right )\|u_{0,2}\|_{L^2}\notag\\
&\qquad\qquad\qquad\qquad +
\frac{C_1}{1-\theta}\gamma(1+T^{1-\theta})\|v\|_{C_{\beta}(T)}
\exp\left (\frac{C_1}{1-\theta}\gamma
T^{1-\theta}\|v\|_{C_{\beta}(T)}\right )\notag\\
&\qquad\qquad\qquad\qquad\qquad\qquad\times\left
(\|u_{0,1}\|_{L^2}+T\|\alpha\|_{L^2}\right ),\tag{\theequation b}
\label{C-alpha-Lambda-1}\\[2mm]
&\|\Lambda_2(w)-\Lambda_2(v)\|_{C([0,T];L^2)}\notag\\
\le& \frac{C_1}{1-\theta}\gamma T^{1-\theta}\exp\left
(\frac{C_1}{1-\theta}\gamma T^{1-\theta}\max
(\|v\|_{C_{\beta}(T)},\|w\|_{C_{\beta}(T)})\right )\|w-v\|_{C_{\beta}(T)}\notag\\
&\qquad\times \left \{\|u_{0,1}\|_{L^2}+T\|\theta\|_{L^2}
+\|v\|_{C_{\beta}(T)}\left (1+
\frac{C_1}{(1-\theta)(2-\theta)}\gamma
T^{2-\theta}\|\beta\|_{L^2}\right )\right \}, \tag{\theequation
c}\label{sup-contr-1}
\end{align}
for any $v,w\in C_{\beta}(T)$. Let us now observe that
$(u_1,u_2,u_3)$ is a classical solution to Problem \eqref{pb-u2-u3}
if and only if $u_3$ is a fixed point of the operator $\Gamma$,
formally defined by
\begin{eqnarray*}
(\Gamma(v))(t,\cdot)=e^{-t\mu_V}e^{t\Delta_2}u_{0,3}
+\int_0^te^{-\mu_V(t-s)}e^{(t-s)\Delta_2}(\Lambda_2(v))(s,\cdot)ds,\qquad\,t\in
(0,T).
\end{eqnarray*}

We are going to prove that the operator $\Gamma$ is a contraction in
${\mathscr Y}_{\beta}(T)=\{u\in C_{\beta}(T):
\|u\|_{C_{\beta}(T)}\le M\}$ provided that $T$ and $M$ are properly
chosen. As a first step, let us prove that $\Gamma$ maps ${\mathscr
Y}_{\beta}(T)$ into itself if $T,M$ are suitably chosen. For this
purpose, we set
\begin{eqnarray*}
K:=\sup_{t>0}\|t^{\beta}e^{t\Delta_2}\|_{D_{\Delta_2}(\beta,\infty)}.
\end{eqnarray*}
Taking \eqref{sup-norm-Lambda-1} into account, we can estimate
\begin{align*}
\|\Gamma&(v)\|_{C_{\beta}(T)}\le K\left \{\|u_{0,3}\|_{L^2}+\frac{T}{1-\beta}\|u_{0,2}\|_{L^2}\right.\\
&\left.+\frac{C_1M}{(1-\beta)(1-\theta)}\gamma T^{2-\theta}\exp\left
(\frac{C_1}{1-\theta}\gamma T^{1-\theta}M\right )\left
(\|u_{0,1}\|_{C_{\beta}(T)}+T\|\alpha\|_{L^2}\right )\right \},
\end{align*}
for any $v\in {\mathscr Y}_{\beta}(T)$. Hence, if we fix
$M>K\|u_{0,3}\|_{L^2}$, we can then choose $T$ small enough such
that $\|\Gamma(v)\|_{C_{\beta}(T)}\le M$ for any $v\in {\mathscr
Y}_{\beta}(T)$. Moreover, taking \eqref{sup-contr-1} into account,
we can estimate
\begin{align*}
\|\Gamma(w)&-\Gamma(v)\|_{C_{\beta}(T)}\le \frac{K}{1-\beta}T
\|\Lambda_2(w)-\Lambda_2(v)\|_{C([0,T];L^2)}\\
& \le \frac{KC_1}{(1-\beta)(1-\theta)}\gamma T^{2-\theta}\exp\left
(\frac{C_1}{1-\theta}\gamma T^{1-\theta}M \right )\|w-v\|_{C_{\theta}(T)}\notag\\
&\qquad\quad\;\times\left\{\|u_{0,1}\|_{L^2}+T\|\alpha\|_{L^2}+M\left
(1+\frac{C_1}{(1-\theta)(2-\theta)}\gamma
T^{2-\theta}\|\alpha\|_{L^2}\right )\right\},
\end{align*}
for any $v,w\in {\mathscr Y}_{\theta}(T)$. This estimate shows that
$\Gamma$ is a $1/2$-contraction provided that $T$ is sufficiently
small. We can thus apply the Banach fixed point theorem and conclude
that there exist $T>0$ and a unique function $u_3\in {\mathscr
Y}_{\beta}(T)$ solving the equation $\Gamma(u_3)=u_3$.

The function $u_3$ actually belongs to $C([0,T];L^2)\cap
C^1((0,T];L^2)\cap C((0,T];H^2_{\sharp})$. Indeed, by
\eqref{C-alpha-Lambda-1}, the function $\Lambda_2(u_3)$ is in
$C^{1-\beta}([0,T];L^2)$. Therefore, \cite[Thm. 4.3.1(i)]{lunardi}
guarantees that the function $\Gamma(u_3)$ has the claimed
regularity properties. Moreover, $D_tu_3=\Delta u_3-\mu_V
u_3+\Lambda_2(u_3)$ in $(0,T]$. As a byproduct, the triplet
$(u_1,u_2,u_3)$ is a classical solution to Problem \eqref{pb-u2-u3}.

By a classical argument we can extend the solution $(u_1,u_2,u_3)$
to a maximal solution defined in some time domain $[0,T_*)$. This
vector valued function (still denoted by $(u_1,u_2,u_3)$) enjoys the
following properties: $u_1,u_2\in C([0,T_*);L^2)\cap
C^1((0,T_*);L^2)$, $u_3\in C([0,T_*);L^2)\cap C^1((0,T_*);L^2)\cap
C((0,T_*);H^2_{\sharp})$.

{\em Step 2.} Here, we prove that if ${\bf u}_0\ge 0$ (where the
inequality is meant componentwise) then the maximal defined solution
${\bf u}$ to Problem \eqref{pb-u2-u3} is non-negative as well in
$(0,T_*)$. Clearly, using formulae \eqref{op-Lambda-1} and
\eqref{operat-Lambda2} it is immediate to check that $u_1$ and $u_2$
are both non-negative whenever $u_{0,1}$ is.

Let us now consider the problem for $u_3$, which we rewrite here:
\begin{equation}
\left\{
\begin{array}{ll}
D_tu_3(t,\cdot)=\Delta u_3(t,\cdot)-\mu_Vu_3(t,\cdot)+N\mu_I
u_2(t,\cdot), &
t\in (0,T_*),\\[2mm]
u_3(0,\cdot)=u_{0,3}.
\end{array}
\right. \label{pb-u3}
\end{equation}
The heat semigroup is positive in $C_{\sharp}$ by the maximum
principle. Since the heat semigroup in $C_{\sharp}$ is the
restriction to $C_{\sharp}$ of the heat semigroup
$\{e^{t\Delta_2}\}$ in $L^2$, by density it follows that
$\{e^{t\Delta_2}\}$ is non-negative as well. This is enough for our
aims. Indeed, the function $u_3$ is given by the variation of
constants formula
\begin{equation}
u_3(t,\cdot)=e^{t(\Delta_2-\mu_V)}u_{0,3}
+N\mu_I\int_0^te^{(t-s)(\Delta_2-\mu_V)}u_2(s,\cdot)ds,\qquad\;\,t\in
(0,T_*),
\label{form-u3}
\end{equation}
$u_{0,3}$ and $u_2$ being non-negative, the function $u_3$ is
non-negative as well.

We have so proved that any solution to Problem \eqref{pb-u2-u3},
corresponding to an initial datum in the first octant, is confined
to the first octant for any $t\in (0,T_*)$. In such a case we can
forget the absolute value in \eqref{pb-u2-u3}.

{\em Step 3.} Here, we prove that any solution ${\bf u}$ to Problem
\eqref{pb-u2-u3}, corresponding to an initial datum ${\bf u}_0\in
{\mathscr D}$, exists for any positive time and it stays bounded.
Here,
\begin{align}
&{\mathscr D}=\{(x,y,z)\in\R^3: 0\le x+y\le M_1,\;\,0\le z\le
M_2\},\nonumber\\
&M_1=\frac{\|\alpha\|_{\infty}}{\min\{\mu_T,\mu_I\}},\qquad
M_2=\frac{M_1N\mu_I}{\mu_V}\|\alpha\|_{\infty}.
\label{choiceM1-M2}
\end{align}
 For this purpose, it is convenient to
introduce the so-called Svab-Zeldovich variable $v=u_1+u_2$. As it
is immediately seen, the function $v$ satisfies the Cauchy problem
\begin{eqnarray*}
\left\{
\begin{array}{ll}
D_tv(t,\cdot)=\alpha-\mu_Tu_1-\mu_Iu_2, & t\in (0,T_*),\\[2mm]
v(0,\cdot)=u_{0,1}+u_{0,2}.
\end{array}
\right.
\end{eqnarray*}
Since $u_1$ and $u_2$ are both positive, then
\begin{equation}
D_tv(t,\cdot)\le\alpha-\min\{\mu_T,\mu_I\}v(t,\cdot):=\alpha-\mu
v(t,\cdot)\le\|\alpha\|_{\infty}-\mu v(t,\cdot),\qquad\,\,t\in
(0,T_*). \label{monk}
\end{equation}
Multiplying both the sides of \eqref{monk} by a non-negative
function $\varphi\in L^2$ and integrating over $\Omega_{\ell}$, one
obtains that the function
$w(t)=\int_{\Omega_{\ell}}v(t,x)\varphi(x)dx$ is in $C^1([0,T_*))$
and solves the differential inequality
\begin{eqnarray*}
D_tw(t)\le \|\alpha\|_{\infty}\|\varphi\|_{L^1(\Omega_{\ell})}-\mu
w(t),\qquad\;\,t\in (0,T_*).
\end{eqnarray*}
Hence,
\begin{eqnarray*}
w(t)\le e^{-\mu t}\left
(\int_{\Omega_{\ell}}v(0,x)\varphi(x)dx+\frac{1}{\mu}(e^{\mu
t}-1)\|\alpha\|_{\infty}\int_{\Omega_{\ell}}\varphi(x)dx \right
),\qquad\,t\in (0,T_*),
\end{eqnarray*}
or, equivalently,
\begin{eqnarray*}
\int_{\Omega_{\ell}}\left (v(t,x)-e^{-\mu
t}v(0,x)-\frac{1}{\mu}(1-e^{-\mu t})\|\alpha\|_{\infty}\right
)\varphi(x)dx\le 0,\qquad\;\,t\in (0,T_*).
\end{eqnarray*}
 From this integral inequality, we can infer that
\begin{equation}
v(t,x)\le e^{-\mu t}v(0,x)+\frac{1}{\mu}(1-e^{-\mu
t})\|\alpha\|_{\infty}\le \frac{1}{\mu}\|\alpha\|_{\infty}=M_1,
\label{stima-v}
\end{equation}
for any $t\in [0,T_*)$ and almost any $x\in \Omega_{\ell}$.

Since $v=u_1+u_2$ and $u_1$ and $u_2$ are both non-negative, it
follows that $u_1$ and $u_2$ can be estimated by the right-hand side
of \eqref{stima-v}.

Finally, let us consider the function $u_3$. From \eqref{form-u3}
and the above results, we can infer that
\begin{eqnarray*}
D_tu_3(t,\cdot)\le\Delta
u_3(t,\cdot)-\mu_Vu_3(t,\cdot)+M_1N\mu_I,\qquad t\in (0,T_*).
\end{eqnarray*}
Note that the function $\overline u_3\equiv M_2$, with $M_2$ being
given by \eqref{choiceM1-M2}, satisfies the previous inequality.
Hence, if $u_{0,3}\le M_2$, then, the solution to Problem
\eqref{pb-u3} is bounded from above by $M_2$. With the previous
choices of $M_1$ and $M_2$, we see that the solution to Problem
\eqref{pb-u2-u3} which corresponds to ${\bf u}_0\in {\mathscr D}$,
stays in ${\mathscr D}$ for any $t\in (0,T_*)$. By virtue of
\cite[Prop. 7.1.8]{lunardi}, ${\bf u}$ can be extended to all the
positive times.

{\em Step 4.} Here, we show that, for any ${\bf u}_0\in {\mathscr
D}$, the solution ${\bf u}$ that we have determined in the previous
steps is, in fact, the unique weak solution to Problem
\eqref{pb-u2-u3} which belongs to $L^2((0,T);L^2)\times
L^2((0,T);L^2)\times L^2((0,T);H^1_{\sharp})$ for any $T>0$. Even if
the following arguments are standard, for the reader's convenience
we go into details.

As a first step, we observe that, since $u_1$, $u_2$ and $u_3$ are
bounded, the weak derivatives $D_tu_1$ and $D_tu_2$ are in
$L^{\infty}((0,+\infty)\times\Omega_{\ell})=L^{\infty}((0,\infty);L^{\infty}(\Omega_{\ell}))$.
Hence, $u_1$ and $u_2$ are locally Lipschitz continuous in
$[0,+\infty)$ with values in $L^{\infty}(\Omega_{\ell})$.

Let us now consider the Cauchy problem for $u_3$ (i.e., problem
\eqref{pb-u3}). Since $u_2$ is Lipschitz continuous in $[0,T]$ with
values in $L^2$, for any $T>0$, by \cite[Thm. 4.3.1(i)]{lunardi},
such a Cauchy problem admits a solution which is in
$C([0,T];L^2)\cap C^1((0,T];L^2)\cap C((0,T];H^2_{\sharp})$. By the
weak maximum principle, the Cauchy problem \eqref{pb-u3} admits a
unique weak solution. Hence, $u_3\in C([0,T];L^2)\cap
C^1((0,T];L^2)\cap C((0,T];H^2_{\sharp})$. Now, we turn back to the
equations for $u_1$ and $u_2$ and conclude that $D_tu_1$ and
$D_tu_2$ are in $C([0,T];L^2)$ for any $T>0$, this implying that
$u_1$ and $u_2$ are in $C^1([0,T];L^2)$. Hence, any weak solution to
Problem \eqref{pb-u2-u3} with data in ${\mathscr D}$ is such that
$u_1,u_2\in C^1([0,T_*);L^2)$ and $u_3\in C([0,T_*);L^2)\cap
C^1((0,T_*);L^2) \cap C((0,T_*);H^2_{\sharp})$. Since we have proved
uniqueness of the solution in this class of functions, uniqueness of
the weak solution follows as well.

Since for non-negative solutions the Cauchy problem \eqref{pb-u2-u3}
coincides with problem \eqref{pde1-ter}-\eqref{pde3-ter}, we have,
thus, established the following:

\begin{proposition} For every ${\bf u_0} \in {\bf H}$, the Cauchy problem
\eqref{pb-u2-u3} possesses a unique solution for all time, ${\bf
u}(t) \in {\bf H}$ for all $t$, ${\bf u} \in L^2(0,T;{\bf V})$ for
all $T>0$. The mapping ${\mathscr S}(t):={\bf u_0} \mapsto {\bf
u}(t)$ is continuous in ${\bf H}$. Furthermore, if ${\bf u_0} \in
{\bf V}$, then $u_3 \in L^2(0,T;H_{\sharp}^2)$.
\end{proposition}

{\em Step 3.} We are now in a position to apply \cite[Thm.
5.1]{marion}. The set
\begin{eqnarray*}
{\mathscr D}=\{(x,y,z)\in\R^3: 0\leq x+y\leq M_1,\; 0\leq z\leq
M_2\},
\end{eqnarray*}
is a positively convex, compact region of $\R^3$. To meet all the
hypotheses of \cite[Thm. 5.1]{marion}, it remains to consider the
non-dissipative part of \eqref{pb-u2-u3}, i.e. the equations for
$u_1$ and $u_2$ which we rewrite in the compact form:
\begin{eqnarray*}
D_t\begin{pmatrix} u_1\\
u_2
\end{pmatrix}
+\begin{pmatrix} \mt+\gamma u_3 & 0 \\ -\gamma u_3 & \mi
\end{pmatrix}\begin{pmatrix} u_1\\
u_2
\end{pmatrix}
+\begin{pmatrix} \alpha\\
0 \end{pmatrix} := G(\cdot,u_3)
\begin{pmatrix} u_1\\
u_2
\end{pmatrix}
+g(\cdot).
\end{eqnarray*}
Obviously the matrix $G(\cdot,u_3)$ has positive eigenvalues
whenever $u_3 \geq 0$, which are bounded from below by positive
constants. Hence,
\begin{eqnarray*}
\langle G(\cdot,u_3)\xi,\xi\rangle \; \geq \mt \|\xi\|^2,
\end{eqnarray*}
for any $\xi\in \R^2$. Thus, condition (4.6) in \cite{marion} is
satisfied. The proof of Theorem \ref{stab-lambda_0-positive} is
completed.
\section{A gamut of some special cases} \label{special}
\setcounter{equation}{0}
\subsection{Numerical illustration (evolution)}\label{numerics2}
We continue the discussion of Subsection \ref{numerics1} in the
framework of the evolution problem
\eqref{pde1-ter}-\eqref{pde3-ter}. In the first case
($\lambda_0<0$), only the non-infected steady state $V_u\equiv 0$
exists. We solve \eqref{pde1-ter}-\eqref{pde3-ter} under particular
initial conditions: we start the infection at the center of the grid
with an inoculum of one viral unit, assuming that $T$ and $I$ are at
their uninfected steady state. One observes that the virus vanishes
very rapidly and the target cells return to their initial level (see
Fig. \ref{evol2} left) in accordance with the stability of the
uninfected equilibrium $V_u$. In the second case ($\lambda_0>0$),
two equilibria exist, $V_u$ and $V_i$. Starting with the same
initial conditions, the virus population grows while the population
of target cells decreases. Both of them achieve an equilibrium
corresponding to the positive infected solution (see Fig.
\ref{evol2} right).
\begin{figure}[htp]
\begin{center}
\epsfig{file=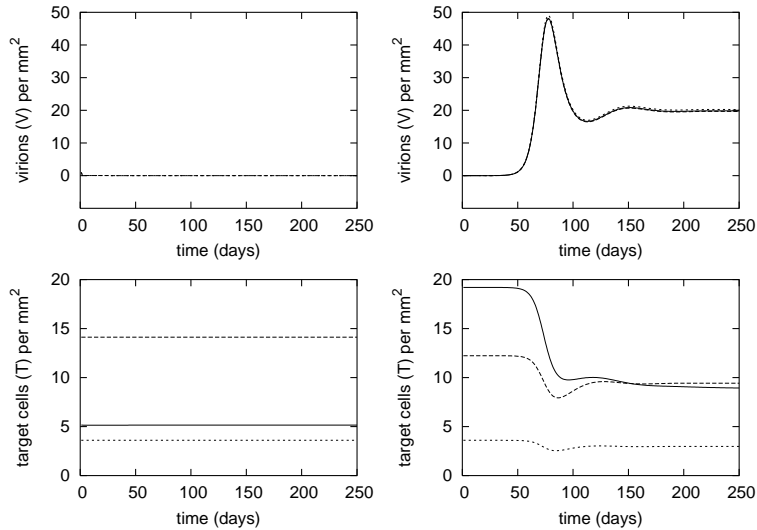,scale=0.8}
\end{center}
\caption{Dynamics of free virus (top) and target cells (bottom) on a
$40 \times 40$ grid. Left: $R_0$ as in Fig. \ref{R012}a,
$\lambda_0<0$ (no infection); right: $R_0$ as in Fig. \ref{R012}b,
$\lambda_0>0$ (infection). The solid, the long and the short dashed
lines depict respectively the dynamics at (10,10), at the site of
viral inoculum (20,20) and at a border site (1,1). Here
$d_V=1,~\ell=1$.} \label{evol2}
\end{figure}%

We are now going to review some particular cases where the stability of
the infected solution is granted.
\subsection{Homogeneous environment with diffusion}
We consider the case when $\alpha$ is a constant such that $R_0$
(see \eqref{R0-def}) is a constant, as well, greater than $1$,
together with $d_V>0$ as in \cite{funk}.

The linearization around $(T_i,I_i,V_i)$ (see \eqref{T-inst-const})
of Problem \eqref{pde1-ter}-\eqref{pde3-ter} is associated with the
linear operator
\begin{eqnarray*}
{\mathscr L}_i=
\begin{pmatrix}
     - \mt R_0 Id & 0 & \displaystyle -\frac{\mv}{N} Id \\[3.5mm]
    \mt(R_0-1) Id& -\mi Id & \displaystyle \frac{\mv}{N} Id\\[3.5mm]
    0 & N \mi Id & d_V\Delta -\mv Id\\
\end{pmatrix}.
\end{eqnarray*}
The same arguments as in the proof of Theorem
\ref{stab-lambda_0-positive}(i) show that the realization $L_i$ of
the operator ${\mathscr L}_i$ in $(L^2)^3$ with domain
$D(L_i)=L^2\times L^2\times H^2_{\sharp}$ generates an analytic
strongly continuous semigroup. We are going to determine its
spectrum. For this purpose we use the discrete Fourier transform.

We consider the realization $\Delta_2$ of the operator $\Delta$ with
domain ${H}_{\sharp}^{2}$. Its real eigenvalues can be labeled as a
non-increasing sequence $(-\lambda_k),~ k=0,1, \ldots$ Only
$\lambda_0=0$ is simple, the other eigenvalues being negative such
that $-\lambda_k \to -\infty$ as $k \to + \infty$.

We claim that the spectrum of the operator $L_i$ is given by
\begin{equation}
\sigma (L_i)= \{-\mu_I,-\mu_TR_0\}\cup\bigcup_{k\in\N\cup\{0\}}
\sigma_k, \label{eq-spectrum}
\end{equation} where $\sigma_k$  is the spectrum
of the matrix
\begin{eqnarray*}
M_k=\begin{pmatrix}
     - \mt R_0 & 0 & \displaystyle -\frac{\mv}{N} \\[3.5mm]
    \mt(R_0-1) & -\mi & \displaystyle \frac{\mv}{N} \\[3.5mm]
    0 & N \mi & -\lambda_k d_V-\mv \\
\end{pmatrix}.
\end{eqnarray*}

To check the claim, let us observe that, if a function ${\bf
v}=(v_1,v_2,v_3)$ in $L^2\times L^2\times H^2$ solves the resolvent
equation $\lambda {\bf u}-L_i{\bf u}={\bf f}$, for some
$\lambda\in\C$ and ${\bf f}=(f_1,f_2,f_3)$ in $(L^2)^3$, then its
Fourier coefficients $(v_{1,k},v_{2,k},v_{3,k})$ ($k=0,1,\ldots$)
solve the infinitely many equations
\begin{eqnarray*}
\begin{pmatrix}
\lambda +     \mt R_0 & 0 & \displaystyle \frac{\mv}{N} \\[3.5mm]
    -\mt(R_0-1) & \lambda+\mi & \displaystyle -\frac{\mv}{N} \\[3.5mm]
    0 & -N \mi & \lambda+\lambda_k d_V +\mv \\
\end{pmatrix}
\begin{pmatrix}
v_{1,k}\\[3.5mm]
v_{2,k}\\[3.5mm]
v_{3,k}
\end{pmatrix}
=
\begin{pmatrix}
f_{1,k}\\[3.5mm]
f_{2,k}\\[3.5mm]
f_{3,k}
\end{pmatrix},
\qquad\;\,k=0,1,\ldots,
\end{eqnarray*}
where $f_{i,k}$ denotes the $k$-th Fourier coefficient of the
function $f_i$ ($i=1,2,3$). Clearly, any eigenvalue of $M_k$
($k=0,1,\ldots$) is an eigenvalue of $L_i$. Therefore,
\begin{eqnarray*}
\sigma (L_i)\supset \bigcup_{k\in\N\cup\{0\}} \sigma_k.
\end{eqnarray*}

On the other hand, if $\lambda\not\in\sigma_k$ for any
$k=0,1,\ldots$, then all the coefficients
$(v_{1,k},v_{2,k},v_{3,k})$ are uniquely determined through the
formulae
\begin{align*}
v_{1,k}=&\frac{1}{{\mathscr P}_k(\lambda)}\Big\{
[(\lambda+\mu_I)(\lambda+\lambda_kd_V+\mu_V)-\mu_I\mu_V]f_{1,k}
-\mu_I\mu_Vf_{2,k}\\
&\qquad\qquad-\frac{\mu_V}{N}(\lambda+\mu_I)f_{3,k}\Big\},\\[2mm]
v_{2,k}=&\frac{1}{{\mathscr P}_k(\lambda)}\Big\{
\mu_T(R_0-1)(\lambda+\lambda_kd_V+\mu_V)f_{1,k} +
(\lambda+\mu_TR_0)(\lambda+\lambda_kd_V+\mu_V)f_{2,k}\\
&\qquad\qquad+\frac{\mu_V}{N}(\mu_T+\lambda)f_{3,k}\Big\},\\[2mm]
v_{3,k}=&\frac{1}{{\mathscr
P}_k(\lambda)}\left\{N\mu_I\mu_T(R_0-1)f_{1,k}+N\mu_I(\lambda+\mu_TR_0)f_{2,k}\right.\\
&\qquad\qquad+\left.(\lambda+\mu_TR_0)(\lambda+\mu_I)f_{3,k}\right\},
\end{align*}
 where
\begin{eqnarray*}
{\mathscr P}_k (\lambda)= \lambda^3 + b_k \lambda^2 + c_k \lambda +
d_k= 0, \quad k=0,1, \ldots
\end{eqnarray*}
and
\begin{align*}
b_k &= \mt R_0 + \mi + \mv +d_V \lambda_k, \\
c_k &= \mt R_0 (\mi + \mv) +d_V\lambda_k(\mt R_0+\mi), \\
d_k &= \mi \mv \mt (R_0 -1) + d_V\lambda_k \mi\mt R_0.
\end{align*}%
Note that, if $\lambda$ differs from both $-\mu_I$ and $-\mu_TR_0$,
then
\begin{eqnarray*}
{\mathscr P}_k(\lambda)\sim
d_V\left\{\lambda^2+\lambda(\mu_TR_0+\mu_I)+\mu_I\mu_TR_0\right\}\lambda_k,\qquad\;\,
{\rm as}~k\to +\infty.
\end{eqnarray*}
Hence, for any
$\lambda\notin\{-\mu_I,-\mu_TR_0\}\cup\bigcup_{k\in\N\cup\{0\}}\sigma_k$,
it holds that
\begin{align*}
v_{1,k} &\sim \frac{\lambda+\mu_I}{\lambda^2+\lambda(\mu_TR_0+\mu_I)+\mu_I\mu_TR_0}f_{1,k},\\[2mm]
v_{2,k}&\sim \frac{\mu_T(R_0-1)f_{1,k}+(\lambda+\mu_TR_0)f_{2,k}}{\lambda^2+\lambda(\mu_TR_0+\mu_I)+\mu_I\mu_TR_0},\\[2mm]
v_{3,k}
&\sim\frac{N\mu_I\mu_T(R_0-1)f_{1,k}+N\mu_I(\lambda+\mu_TR_0)f_{2,k}
+(\lambda+\mu_TR_0)(\lambda+\mu_I)f_{3,k}}
{d_V\{\lambda^2+\lambda(\mu_TR_0+\mu_I)+\mu_I\mu_TR_0\}}\frac{1}{\lambda_k},
\end{align*}
as $k\to +\infty$. It follows that the sequences $\{v_{1,k}\}$,
$\{v_{2,k}\}$ and $\{\lambda_kv_{3,k}\}$ are in $\ell^2$. This shows
that the series whose Fourier coefficients are $v_{1,k}$, $v_{2,k}$
and $v_{3,k}$, respectively, converge in $L^2$ (the first two ones)
and in $H^2_{\sharp}$ (the latter one). The inclusion
\begin{eqnarray*}
\sigma (L_i)\subset
\{-\mu_I,-\mu_TR_0\}\cup\bigcup_{k\in\N\cup\{0\}} \sigma_k
\end{eqnarray*}
follows.
We now observe that a straightforward computation shows
that $\lambda=-\mu_TR_0$ is in the essential spectrum of $L_i$. Also
$\lambda=-\mu_I$ belongs to the essential spectrum of $L_i$ and, in
the case when $\mu_I=\mu_T$, it belongs also to the point spectrum.
The set equality \eqref{eq-spectrum} is proved.

Clearly $\sigma_k$ has three eigenvalues (counted with their
multiplicity), either all real, or one real and two complex
conjugates. Routh-Hurwitz criterion enables us to determine whether
the elements of $\sigma_k$  have negative real parts. The latter
holds if and only if $b_k$, $d_k$ and $b_k c_k-d_k$ are positive,
which is clearly true whenever $R_0>1$.

\begin{remark}
As it is easily seen $\sigma_0$ is the spectrum of
the operator
\begin{eqnarray*}
\begin{pmatrix}
     - \mt R_0 Id & 0 & \displaystyle -\frac{\mv}{N} Id \\[3.5mm]
    \mt(R_0-1) Id& -\mi Id & \displaystyle \frac{\mv}{N} Id\\[3.5mm]
    0 & N \mi Id & -\mv Id\\
\end{pmatrix},
\end{eqnarray*}
which is associated with the linearization at $(T_i,I_i,V_i)$ of
Problem \eqref{ode1}-\eqref{ode3}. Since, as we have already
remarked, the spectrum of $L_i$ is the union of the sets $\sigma_k$
($k\in\N\cup\{0\}$) and the points $-\mu_I$, $-\mu_TR_0$, the
scenario is one of the following:
\begin{enumerate}[(a)]
\item
the diffusion does not improve the stability of the solution $(T_i,I_i,V_i)$.
Therefore, the stability issue is identical to that of the ODE system ($d_V=0$);
\item
the diffusion worsen the stability of the solution $(T_i,I_i,V_i)$.
\end{enumerate}
Therefore, we are unable to confirm Funk et al. \cite{funk}, who
pointed out that the presence of a spatial structure enhances
population stability with respect to non-spatial models. Only some
smoothing effect can be credited to the diffusion.
\end{remark}

\subsection{Death rates $\mi \leq \mt$, $\lambda_0>0$}
This case leads to a mathematically interesting framework although it has
little biological relevance, since $\mi > \mt$ in the literature (e.g., $\mt=0.01,~\mi=0.39$
in \cite{ciupe}). For the latter reason we will not elaborate the case extensively.

Using the Svab-Zeldovich variable $v=u_1+u_2$ we can transform
Problem \eqref{pde1-ter}-\eqref{pde3-ter} into the following
equivalent one for the unknowns $u_1$, $v$ and $u_3$:
\begin{equation}
\left\{
\begin{array}{ll}
D_tu_1(t,\cdot)=-\mu_Tu_1(t,\cdot)-\gamma
u_1(t,\cdot)u_3(t,\cdot)+\alpha(\cdot),
& t>0,\\[2mm]
D_tv(t,\cdot)=\alpha-\mu_I v(t,\cdot) -(\mu_T-\mu_I) u_1(t,\cdot),
& t>0,\\[2mm]
D_tu_3(t,\cdot)=\Delta u_3(t,\cdot)-\mu_Vu_3(t,\cdot)+N\mu_I
v(t,\cdot)-N\mu_Iu_1(t,\cdot),\q & t>0,\\[2mm]
u_i(0,\cdot)=u_{0,i},\;\; i=1,3,\\[2mm]
v(0,\cdot)=u_{0,1}+u_{0,2}.
\end{array}
\right. \label{pb-u2-u3-zeldovich}
\end{equation}
It is not difficult to see that the mapping $u_3 \mapsto u_1$ is
non-increasing, so is the mapping $u_1 \mapsto v$ thanks to the
hypothesis $\mu_T-\mu_I>0$. Hence, the mapping $u_3 \mapsto v$ is
non-decreasing. Finally, the mapping $u_3 \mapsto u_2=v-u_1$ is
non-decreasing. Based on these observations, following \cite{pao} it
is possible to construct two sets of monotone sequences  which
converge to the solution of \eqref{pb-u2-u3-zeldovich}. These
sequences start respectively from upper and lower solutions defined
as in Section \ref{sec-steady}, to stay away from the trivial
solution. It is well-known that a solution of an evolution problem
constructed via such a monotone sequence scheme, with suitable
initial conditions between upper and lower solutions, achieves a
stable equilibrium (see \cite{pao}). Therefore, the infected
solution is asymptotically stable. Numerical computations in the
phase plan (see Fig. \ref{phase}) illustrate the difference in the
virus dynamics when $\mi < \mt$ (monotonicity) and $\mi > \mt$
(spirals).
\begin{figure}[htp]
\begin{center}
\begin{tabular}{c}
\epsfig{file=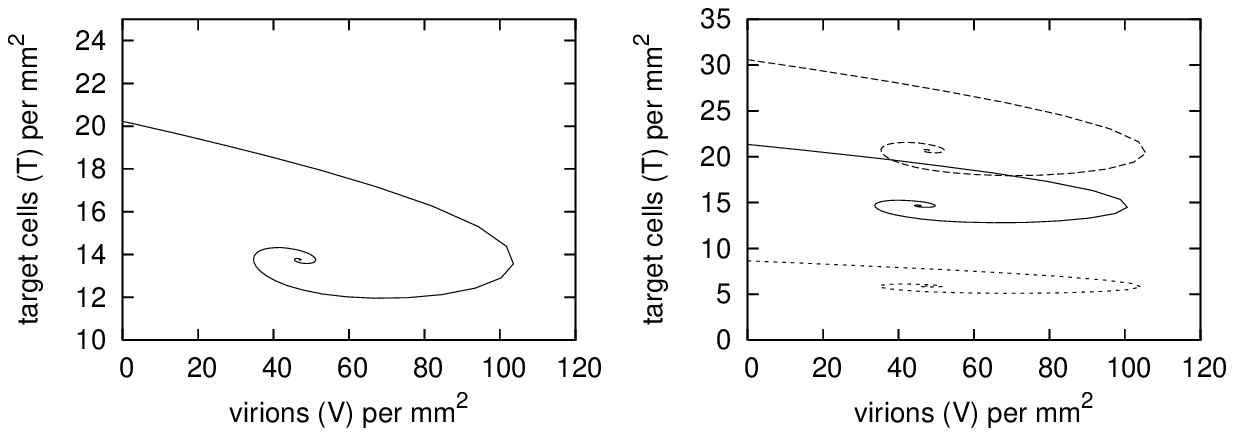}\\
(a) \Large{$\mi>\mt$}\\
\epsfig{file=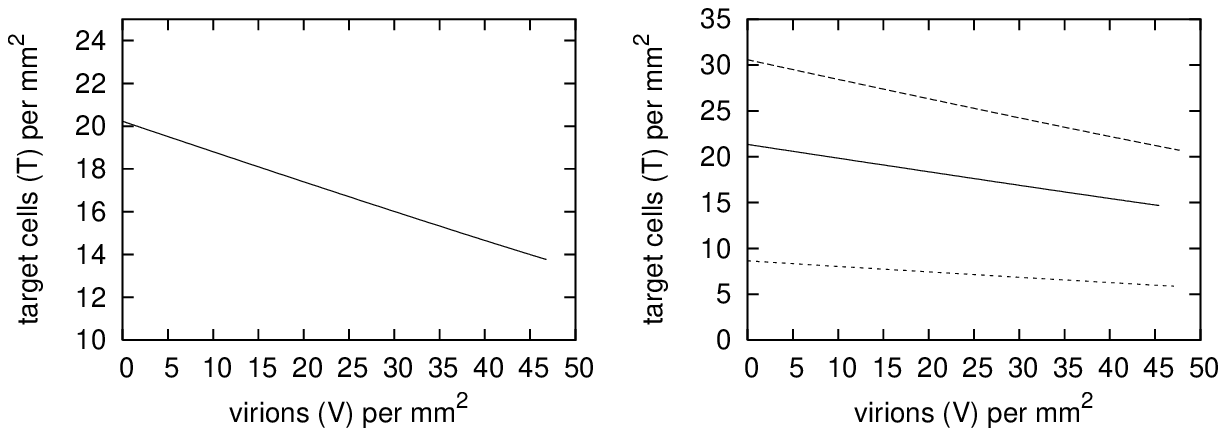}\\
(b) \Large{$\mi<\mt$}\\
\end{tabular}
\end{center}
\caption{Virus vs. target cells in the phase plane, in the case of
infection. Left: mean values. Right: the solid, the long and the
short dashed lines depict, respectively, the values at (10,10), at
the site of viral inoculum (20,20), and at a border site (1,1). The
parameters have the following values:
$\gamma=0.001,~N=1000,~\mt=0.1,~\mv=10,~d_V = 1,~\ell = 1$; $\alpha$
varies between $0.1$ and $5.0$; $\mi=1$ in (a), $\mi=0.01$ in (b).}
\label{phase}
\end{figure}
\subsection{Quasi-steady problem} In this part we assume that $T$ and $I$ are at their
equilibrium. In such a case,
\eqref{pde1-ter}-\eqref{pde3-ter} reads
\begin{align*}
&0=\alpha - \gamma V T - \mu_T T,\\
&0= \gamma V T - \mu_I I,
\\
&V_t=N \mi I- \mu_V V + d_V \Delta V,
\end{align*}
which is equivalent to the scalar parabolic equation for $V$ only,
with periodic boundary conditions:
\begin{equation}
V_t=d_V \Delta V - \mv V + \mt\mv R_0\frac{V}{\gamma V +\mt}.
\label{semil-evol}
\end{equation}
The latter is the natural evolution problem associated with
\eqref{semil}. It is clear that \eqref{semil-evol} has the same
non-positive equilibria, namely $V_u=0$ and $V_i>0$, in the case
when $\lambda_0>0$. The stability of $V_i$ can be proved according
to \cite[Sec. 5.3]{henry} by constructing a Lyapunov function.

\section{Homogenization}
\label{sec-5}
\setcounter{equation}{0}
This section is concerned
with the case wherein the environment is heterogeneous and is formed
of rapidly alternating sinks and sources. For a fixed integer $k$,
we imagine that $\overline{\Omega_{\ell}}$ is divided into a network
of $k^2$ periodic squares $\Omega_{\varepsilon}$, where
$\varepsilon=\ell/k$. The heterogenous reproductive ratio  $R_0$
will depend upon $\e$, see Fig. \ref{R03}b. The idea of
homogenization is to let $\e \to 0$ and find the equivalent
homogenized medium. Therefore, such a heterogenous environment can
be replaced by its homogenized limit for easier computations and
analysis.

\begin{figure}[htp]
\begin{center}
\begin{tabular}{cc}
\epsfig{figure=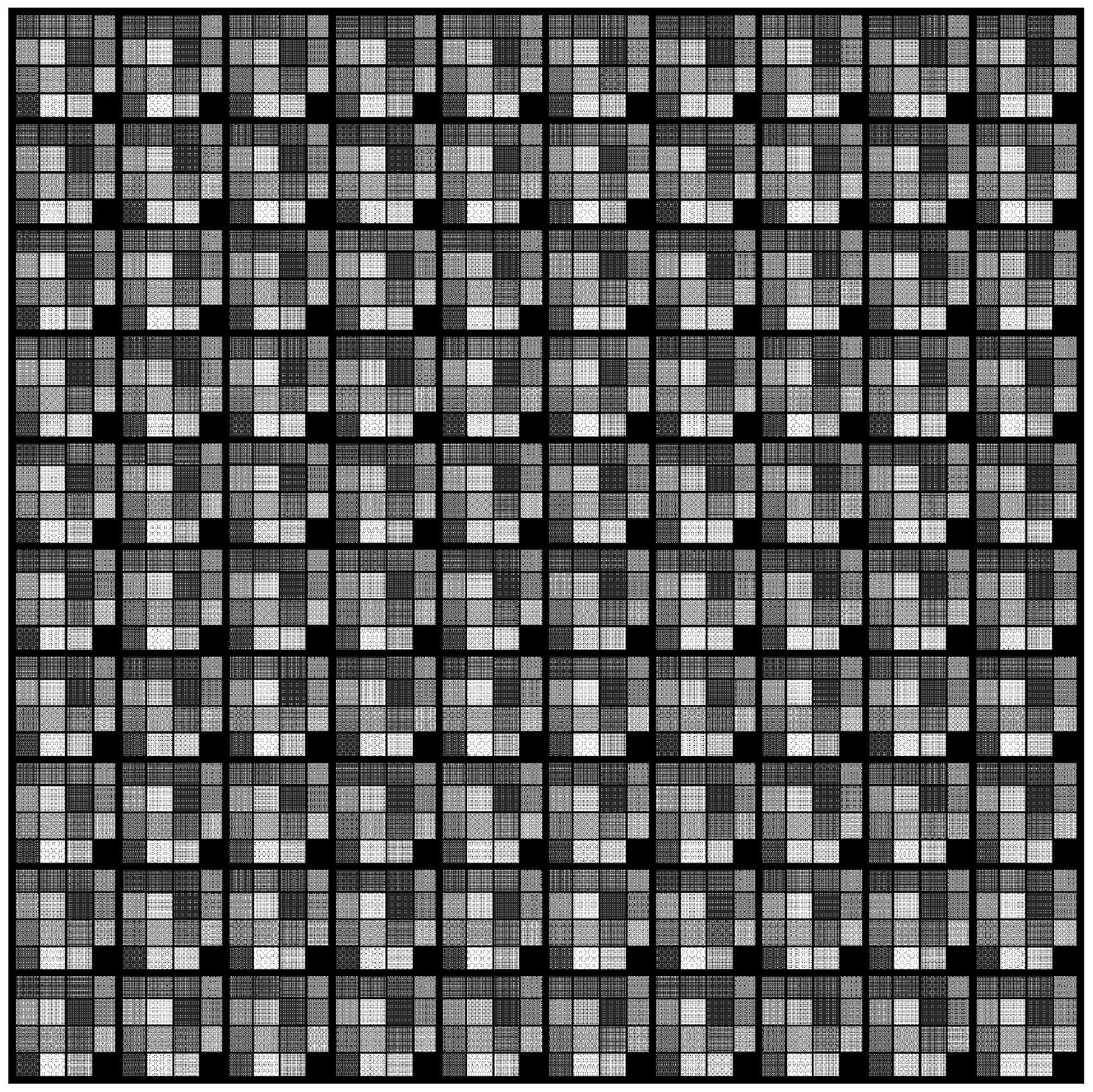,width=6.8cm,angle=0} & \epsfig{figure=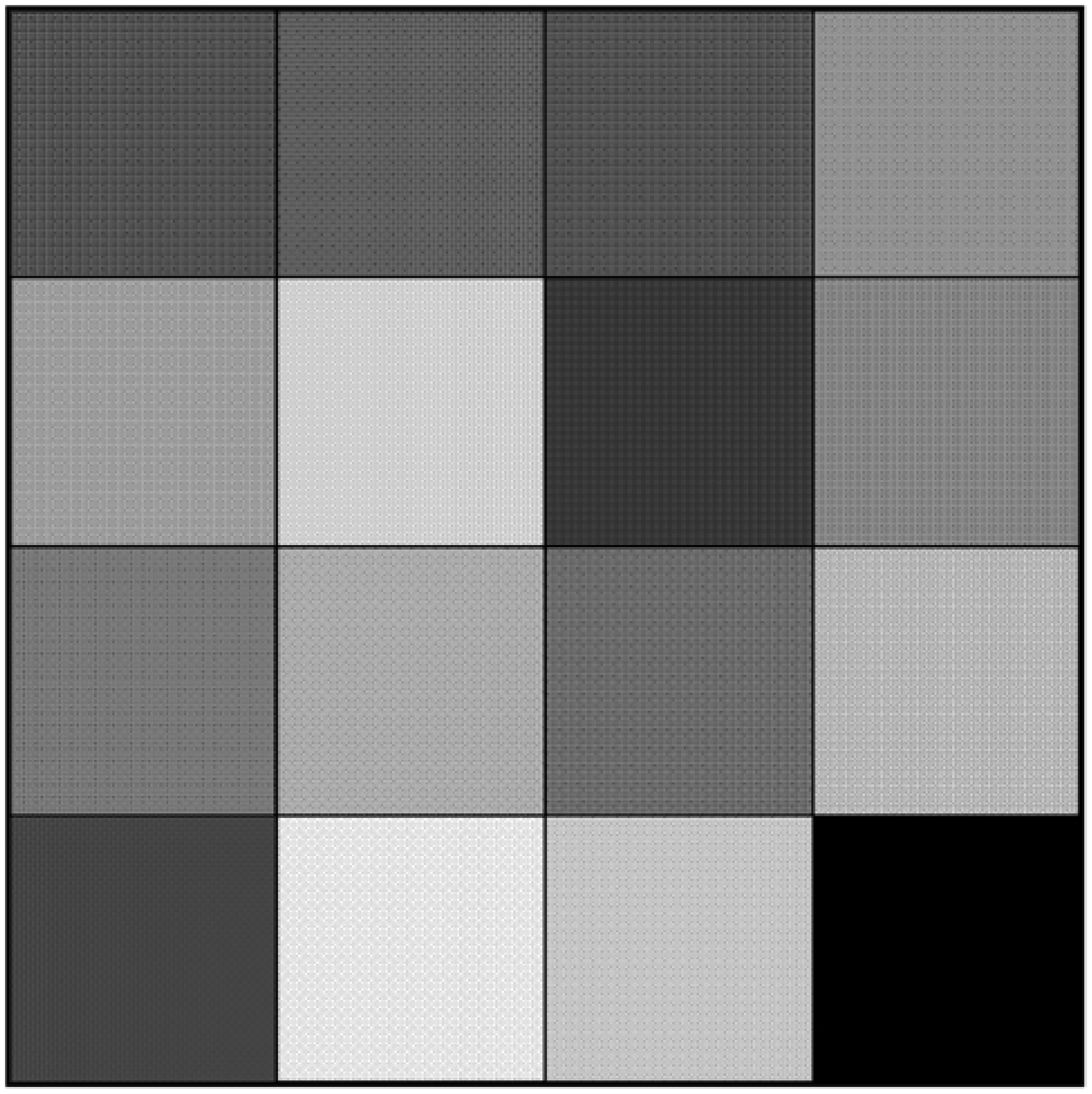,width=3.3cm,angle=0} \\
(a) & (b)\\
\end{tabular}
\end{center}
\caption{Periodic structure with $\ell=1$ (a): $\Omega_{1}$ is
divided in a network of $40 \times 40$ sites. The sites are pieced
together in $4 \times 4$ periodic squares $\Omega_{\e}$ with
$\e=0.1$. (b): zoom of $\Omega_{\e}$ which contains $16$ sites. The
darkness represents the scale of $R_0$ on each site (the darker the
square, the larger is $R_0$). \label{R03}}
\end{figure}

More precisely, we introduce a normalized periodic function $R_0$
as a function of the variable $y=(y_1,y_2)$,
of period $(1,1)$, and we define:
\begin{eqnarray*}
R^{\e}_0(x_1,x_2)=R_0\left (\frac{x_1}{\varepsilon}, \frac{x_2}{\varepsilon}\right ).
\end{eqnarray*}
\begin{remark}
$x=(x_1,x_2)$ is the macroscopic variable while $y=(y_1,y_2)$ is the
microscopic one.
\end{remark}
We consider the problem:
\begin{equation}
\label{semil-e} d_V \Delta V^{\e} - \mv V^{\e} = -\mt\mv
R^{\e}_0\frac{V^{\e}}{\gamma V^{\e} +\mt},
\end{equation}
on $\Omega_{\ell}$ with periodic boundary conditions as above. The
idea is to find the limiting {\it homogenized equation} as $\e \to
0$. We start with the following lemma:

\begin{lemma}\label{lem-4.2} Let $R_0:\R^2\to\R$ be a periodic $($with period $1$
in each variable$)$, piecewise continuous function. For $x \in \R^2$
and  $\e>0$, set $R_0^{\e}(x)= R_0(\frac{x}{\e})$. Then, the
following properties are met.
\begin{enumerate}[\rm (i)]
\item
$R_0^{\e}$ tends to ${\mathscr M}_y (R_0)$ as $\varepsilon\to 0$, weakly
in $L^p_{\rm loc}(\R^2)$ for any $p\geq 1$;
\item
For any $\e>0$, set
\begin{equation}
\label{variat-e} -\lambda_0^{\e} = \inf_{\psi \in {H^1_{\sharp}},
\psi \not\equiv
0}\left\{\frac{d_V\int_{\Omega_{\ell}}|\nabla\psi|^2dx +
\mv\int_{\Omega_{\ell}}(1-R_0^{\e})\psi^2 dx}
{\int_{\Omega_{\ell}}\psi^2 dx} \right\}.
\end{equation}
Then, $\lambda_0^{\varepsilon} \to \mv({\mathscr M}_y (R_0) -1)$ as
$\varepsilon \to 0$.
\end{enumerate}
\end{lemma}
\begin{proof}
Property $(i)$ follows straightforwardly from e.g., \cite[p.
5]{JKO}. Thanks to $(i)$, one can take the limit as $\e\to 0$ in the
right-hand side of \eqref{variat-e} and show that $\lambda_0^{\e}$
tends to the largest eigenvalue of
\begin{eqnarray*}
d_V\Delta  - \mu_V(1-{\mathscr M}_y (R_0))Id,
\end{eqnarray*}
on $\Omega_{\ell}$ with periodic boundary conditions. Let us prove
this claim. As a first step, we observe that there exist two
constants $C_1$ and $C_2$ such that $C_1\le\lambda_0^{\e}\le C_2$,
since the function $R_0^{\e}$ is bounded.

Next, $\lambda_0^{\e}$ is the largest eigenvalue of the Sturm-Liouville
eigenvalue problem
\begin{eqnarray*}
d_V\Delta\psi - \mu_V(1-R_0^{\e})\psi =\lambda \psi,
\end{eqnarray*}
with periodic boundary conditions we documented in Theorem
\ref{sturm-eigen}, associated with the eigenfunction $\psi^{\e}$. We
may assume that $\int_{\Omega_{\ell}} (\psi^{\e})^2 dx =1$. As we
pointed it out in Theorem \ref{sturm-eigen}, $\psi^{\e}$ does not
change sign.

It is clear that $\psi^{\e}$ is bounded in $H_{\sharp}^2$. Then,
there exists an infinitesimal sequence $\{\e_n\}$ such that
$\lambda_0^{\e_n} \to \lambda_0^{0}$, $\psi^{\e_n}\to \psi^{0}$
weakly in $H_{\sharp}^2$, strongly in $H_{\sharp}^{2-\eta}$ and
(hence) uniformly in ${\overline \Omega_{\ell}}$. Note that
$\int_{\Omega_{\ell}} (\psi^{0})^2 dx =1$ and $\psi^{0}$ does not
change sign.

Since
\begin{eqnarray*}
d_V\Delta\psi^{\e_n} - \mu_V(1-R_0^{\e_n})\psi^{\e_n}
=\lambda^{\e_n} \psi^{\e_n},
\end{eqnarray*}
it is not to difficult to pass to the limit in the above equation as
$n\to +\infty$ in the distributional sense and see that
\begin{eqnarray} \label{limit-eigen}
d_V\Delta\psi^{0} - \mu_V(1-{\mathscr M}_y (R_0))\psi^{0} =\lambda^{0}_0 \psi^{0},
\end{eqnarray}
with periodic boundary conditions. Therefore, $\lambda_{0}^{0}$ is
an eigenvalue of the operator $d_V\Delta- \mu_V(1-R_0^{\e})Id$,
associated with the eigenfunction $\psi^{0}$. Since $\psi^{0}$ does
not change sign, $\lambda^{0}_0$ is the largest eigenvalue of
\eqref{limit-eigen}. Obviously, $\psi^{0}$ is a constant, hence
$\lambda^{0}_{0}$ is explicit. Finally, checking that all the
sequence $\lambda_0^{\e}$ converges to $\lambda_0^0$, as $\e \to 0$,
is an easy task.
\end{proof}

Next, we prove the following result.
\begin{theorem}
Assume that ${\mathscr M}_y (R_0) > 1$. Then, there exists $\e_0$
such that, for any $\e\in (0,\e_0]$, the equation \eqref{semil-e}
has a positive solution $V^{\e}\in H^2_{\sharp}$. As
$\varepsilon\to 0$, $V^{\e}$ tends to
\begin{eqnarray*}
V^0=\frac{\mt}{\gamma}({\mathscr M}_y (R_0)-1),
\end{eqnarray*}
in $H^1_{\sharp}$ and uniformly in $\overline{\Omega_{\ell}}$.
\end{theorem}
\begin{proof}
As a first step, we observe that, since ${\mathscr M}_y (R_0) > 1$,
from Lemma \ref{lem-4.2}(ii) it follows that there exist
$\delta,~\varepsilon_0>0$ such that $\lambda_0^{\e}>\delta$ for
$\e\in (0,\varepsilon_0]$, and \eqref{semil-e} has a unique positive
solution $V^{\e}$. Clearly, $\Delta V^{\e}$ is bounded in
$L^{\infty}(\Omega_{\ell})$ and this implies that $V^{\e}$ is
bounded in $H^2_{\sharp}$. Arguing as in the proof of Lemma
\ref{lem-4.2}, one can extract an infinitesimal sequence $\{\e_n\}$
such that $V^{\e_n}$ converges strongly in $H^{2-\eta}$ to a limit
$V^0\in H^2_{\sharp}$, which verifies the equation
\begin{equation}
\label{limit-1}
d_V \Delta V^0- \mv V^0 = -\mt\mv {\mathscr M}_y(R_0)\frac{V^0}{\gamma V^0 +\mt}.
\end{equation}
Because of the periodic boundary conditions, Equation
\eqref{limit-1} has only constant solutions. Therefore,
\begin{eqnarray*}
- \mv V^0 = -\mt\mv {\mathscr M}_y (R_0)\frac{V^0}{\gamma V^0 +\mt},
\end{eqnarray*}
and it only remains to prove that $V^0$ is not the trivial solution.
With obvious notations, we recall (see the proof of Theorem
\ref{thm-2.4}) that $V^{\e} \geq
\underline{v}_0^{\e}=c^{\e}\varphi_0^{\e}$ where $\varphi_0^{\e}$ is
the positive eigenfunction associated with the largest eigenvalue
$\lambda_0^{\e}$ and
\begin{eqnarray*}
c^{\e}= \min \left\{\frac{\lambda_0^{\e}}{\gamma\mu_V{\mathscr R}},
\frac{\mt({\mathscr R} -1)}{\gamma} \right\} \geq \min \left\{\frac{\delta}{\gamma\mu_V{\mathscr R}},
\frac{\mt({\mathscr R} -1)}{\gamma} \right\}.
\end{eqnarray*}
Since $\max \varphi_0^{\e} =1$, it is clear that $V^{\e}$ remains bounded away from $0$ whenever $\e \leq \e_0$.

Finally, checking that  $V^{\e}$ itself converges to $V^0$ as $\e\to
0$ is immediate. This concludes the proof.
\end{proof}

\subsection{Numerical illustration (homogenization)}
We consider a model where $R_0$ is as in Fig. \ref{R03}, taking its
values on the elementary $4 \times 4$ grid (b) as in Tab. \ref{R0}.
\begin{table}[!h]
\begin{center}
\begin{tabular}{|c|c|c|c|c|}
\hline
\mbox{$j$}& \multicolumn{1}{|p{1.5cm}|}{\centering $R_0(1,j)$}
& \multicolumn{1}{|p{1.5cm}|}{\centering $R_0(2,j)$} & \multicolumn{1}{|p{1.5cm}|}{\centering $R_0(3,j)$}& \multicolumn{1}{|p{1.5cm}|}{\centering $R_0(4,j)$} \\
\hline \hline
$1$ & $1.60$\phantom{$0$}  & $1.41$\phantom{$0$E$-4$}        & $1.55$\phantom{$0$}  & $0.819$ \\
$2$ & $0.800$  & $0.165$\phantom{E$-4$}       & $2.59$\phantom{$0$}  & $0.872$ \\
$3$ & $1.20$\phantom{$0$}  & $0.489$\phantom{E$-4$}       & $1.37$\phantom{$0$}  & $0.453$ \\
$4$ & $2.09$\phantom{$0$}  & $4.25$E$-4$\phantom{$0$}    & $0.270$ & $2.80$\phantom{$0$} \\
\hline
\end{tabular}
\\[6mm]
\end{center}
\caption{Values of $R_0(i,j)$, $i,j=1,\ldots,4$, on $\Omega_{\e}$ as
in Fig. \ref{R03}b, to be viewed as a $4 \times 4$ matrix.}
\label{R0}
\end{table}

It is easy to compute the mean value ${\mathscr
M}_y (R_0)=1.16$ and the homogenized viral density $V^0=16.1$.
Fig. \ref{homog} shows how the virus $V^{\e}$ at $\e=0.1$ (left) oscillates slightly around its
homogenized limit $V^0$ (right).
\begin{figure}[htp]
\begin{center}
\epsfig{file=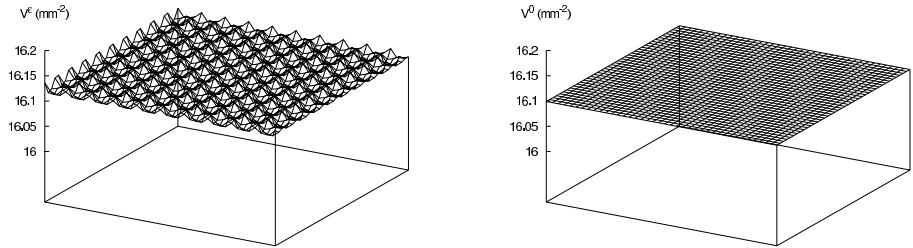,scale=0.525}
\end{center}
\caption{$V^{\e}$ (left) vs. $V^0$ (right). The grid and parameters are as in Fig. \ref{R03}
and Tab. \ref{R0}; $\e=0.1,~d_V = 1,~\ell = 1$.}
\label{homog}
\end{figure}

\appendix
\setcounter{equation}{0}
\section{Proof of Theorem \ref{sturm-eigen}}
\label{app-A}
\begin{proof}
It is well known that the realization $\Delta_2$ of the Laplacian in
$L^2$, with domain $H^2_{\sharp}$, is a sectorial operator. Since
${\mathscr A}$, is a bounded perturbation of the Laplacian,
${\mathscr A}$ is sectorial as well, and, hence, its spectrum is not
empty. Let us fix $\lambda_0\in\rho({\mathscr A})\cap
(\mu_{\infty},+\infty)$ such that $[\lambda_0,+\infty)$ is in the
resolvent set of $T$. Here, by $\mu_{\infty}$ we denote the sup-norm
of the function $\mu$. The operator ${\mathscr A}_0:={\mathscr
A}-\lambda_0I$ turns out to be invertible and its resolvent set
contains $[0,+\infty)$. Since the operator $T_0^{-1}$ is continuous
from $L^2$ into $H^2_{\sharp}$, it is continuous, in particular,
from $H^1_{\sharp}$ into itself, when this latter space is endowed
with the inner product
\begin{align*}
\langle v,w\rangle&=
d\int_{\Omega_{\ell}}(D_{x_1}vD_{x_1}w+D_{x_2}vD_{x_2}w
)dx+\int_{\Omega_{\ell}}(\mu+\lambda_0) vwdx\\
&:=d\int_{\Omega_{\ell}}\nabla v\cdot\nabla
wdx+\int_{\Omega_{\ell}}(\mu+\lambda_0) vwdx,
\end{align*}
which is equivalent to the Euclidean inner product of $H^1$. Since
$H^2_{\sharp}$ is compactly embedded into $H^1_{\sharp}$ (see e.g.,
\cite[Thm. 3.7]{agmon}), the operator ${\mathscr A}_0^{-1}$ is
compact from $H^1_{\sharp}$ into itself. Moreover, ${\mathscr
A}_0^{-1}$ is self-adjoint in $H^1_{\sharp}$. Indeed,
\begin{align}
\langle {\mathscr A}_0^{-1}u,v\rangle =&d\int_{\Omega_{\ell}}\nabla
{\mathscr A}_0^{-1}u\cdot\nabla
vdx+\int_{\Omega_{\ell}}(\mu+\lambda_0) {\mathscr A}_0^{-1}u\,vdx\nonumber\\
= &-d\int_{\Omega_{\ell}} \Delta {\mathscr A}_0^{-1}u\,vdx +
\int_{\Omega_{\ell}}(\mu+\lambda_0) {\mathscr A}_0^{-1}u\,vdx\nonumber\\
= &-\int_{\Omega_{\ell}}uvdx, \label{compact}
\end{align}
for any $u,v\in L^2$. Now, from the general theory of self-adjoint
compact operators, it follows that the spectrum of ${\mathscr
A}_0^{-1}$ consists of a sequence of real eigenvalues which
converges to $0$. As a byproduct, the spectrum of ${\mathscr A}_0$
consists of a sequence of eigenvalues diverging to $-\infty$. More
precisely, $\sigma({\mathscr A}_0)\subset (-\infty,0)$ and
$\lambda\in \sigma({\mathscr A}_0)$ if and only if $\lambda^{-1}$ is
in $\sigma({\mathscr A}_0^{-1})$. In particular, the maximum
eigenvalue of ${\mathscr A}_0$ is the inverse of the minimum
eigenvalue of ${\mathscr A}_0^{-1}$. Since ${\mathscr A}_0^{-1}$ is
a compact operator, its minimum eigenvalue is defined by
\begin{eqnarray*}
\lambda_{\min}({\mathscr A}_0^{-1})=\inf_{\psi\in
H^1_{\sharp},\,~\psi\neq 0} \left\{\frac{\langle {\mathscr
A}_0^{-1}\psi,\psi\rangle}{\langle\psi,\psi\rangle}\right\}.
\end{eqnarray*}
Taking \eqref{compact} into account, we can estimate
\begin{align*}
\lambda_{\min}({\mathscr A}_0^{-1})& = \inf_{\psi\in
H^1_{\sharp},\,~\psi\neq 0}\left\{
\frac{-\int_{\Omega_{\ell}}\psi^2dx}{d_V\int_{\Omega_{\ell}}|\nabla_x\psi|^2dx
+\int_{\Omega_{\ell}}(\mu+\lambda_0)\psi^2dx}\right\}\\
& = -\sup_{\psi\in H^1_{\sharp},\,~\psi\neq 0}\left\{
\frac{\int_{\Omega_{\ell}}\psi^2dx}{d_V\int_{\Omega_{\ell}}|\nabla_x\psi|^2dx
+\int_{\Omega_{\ell}}(\mu+\lambda_0)\psi^2dx}\right\}\\
& =-\left (\inf_{\psi\in H^1_{\sharp},\,~\psi\neq 0}\left\{
\frac{d\int_{\Omega_{\ell}}|\nabla_x\psi|^2dx
+\int_{\Omega_{\ell}}\mu\psi^2dx}{\int_{\Omega_{\ell}}\psi^2dx}\right\}
+\lambda_0\right )^{-1}.
\end{align*}
Formula \eqref{variat} follows at once, observing that
$\lambda_{\max}({\mathscr A})=\lambda_0+\lambda_{\max}({\mathscr
A}_0)$.

The last assertion of the theorem follows from the Krein-Rutman
Theorem applied to the restriction of ${\mathscr A}_0$ to the space
$C_{\sharp}(\R^2)$ (of all functions $f:\R^2\to\R$ which are
continuous with period $\ell$ in each variable), via the maximum
principle (see e.g., \cite{schaefer}). Indeed, since $H^2_{\sharp}$
is continuously embedded into $C_{\sharp}(\R^2)$, the restriction
$({\mathscr A}^{-1}_0)_{| C_{\sharp}(\R^2)}$ of the operator
${\mathscr A}_0^{-1}$ to $C_{\sharp}(\R^2)$ is compact from
$C_{\sharp}(\R^2)$ into itself. Moreover, it is clear that
${\mathscr A}^{-1}_0$ and $({\mathscr A}^{-1}_0)_{|
C_{\sharp}(\R^2)}$ have the same eigenvalues. Let now $f$ be a
non-negative (non trivial) function in $C_{\sharp}(\R^2)$. Then, the
function ${\mathscr A}^{-1}_0f$ is in $H^2_{\sharp}$. Hence, in
particular, it belongs to $H^2((-\ell,2\ell)\times (-\ell,2\ell))$
and solves the equation $(\mu+\lambda_0) u-d\Delta u=f$. By the
classical maximum principle, $u$ is non-negative in
$(-\ell,2\ell)\times (-\ell,2\ell)$. Actually $u$ is everywhere
positive. Indeed if $u(x_0)=0$ at some point $x_0\in
(-\ell,2\ell)\times (-\ell,2\ell)$, then, still by the maximum
principle, it would follow that $u\equiv 0$ in $[-\ell,2\ell]\times
[-\ell,2\ell]$, which clearly cannot be the case.
\end{proof}

\section{A maximum principle}

\begin{proposition}\label{max-princ}
Let ${\mathscr L}$ be a second order operator with constant
coefficients. Let $u\in H^2_{\sharp}$ satisfy the inequality
${\mathscr L}u\le 0$. Then, $u\ge 0$. Similarly, if $u$ belongs to $
C^1((0,T);C(\overline{\Omega_{\ell}}))\cap C([0,T];H^2_{\sharp})$,
is such that ${\mathscr L}u\in C([0,T]\times \Omega_{\ell})$, and it
satisfies the differential inequalities $D_tu-{\mathscr L}u\ge 0$
and $u\le 0$ in $[0,T]\times\partial \Omega_{\ell}$, then, $u\le 0$
in $[0,T]\times \Omega_{\ell}$.
\end{proposition}

\begin{proof}
For the reader's convenience, we sketch the proof of the second
statement, the first one being a particular case of the second one.
Since $H^2_{\sharp}$ continuously embeds in the set of all
continuous functions which are periodic, with period $\ell$ with
respect to all the variables, then $u$ can be extended by
periodicity with a function (still denoted by $u$) which is
continuous in $[0,T]\times\R^N$ and is therein continuously
differentiable with respect to the time variable.

Suppose by contradiction that $u$ is not everywhere non-positive in
$[0,T]\times\R^N$. Then, $u$ has a negative minimum at some point
$(t_0,x_0)\in (0,T_0]\times (-\ell,2\ell)^2$. Then, clearly,
$D_tu(t_0,x_0)\le 0$. Moreover, since ${\mathscr L}u$ is a
continuous function, ${\mathscr L}u(t_0,x_0)\ge 0$. The classical
maximum principle yields the assertion.
\end{proof}

\section*{Acknowledgment} One of the authors (L.L.) greatly
acknowledges the Institute of Mathematics of the University of Bordeaux I
for the warm hospitality during his visit as an invited professor (2008-2009).

\end{document}